\documentclass[runningheads]{llncs}
\usepackage{amssymb}
\usepackage{url}
\usepackage{hyperref}
\usepackage{graphicx}
\usepackage{amsfonts}
\usepackage{amsxtra}
\usepackage{algorithm}
\usepackage{algorithmic}
\usepackage{caption}
\usepackage{makeidx}
\usepackage{bbm}
\usepackage{subfigure}

\usepackage{picins}
\usepackage{mathtools}
\mathtoolsset{showonlyrefs} 
\usepackage{tikz}

\newcommand*{\R}{{\mathbb R}}

\newcommand*{\e}{\varepsilon}
\newcommand*{\la}{\langle}
\newcommand*{\ra}{\rangle}

\def\C{\mathcal C}

\def\U{\mathcal{U}}
\newcommand*{\one}{\mathbbm{1}}

\def\bld{\boldsymbol}

\newcommand*{\Bpi}{\bld\pi}

\renewcommand{\u}{\mathtt{u}}
\renewcommand{\v}{\mathtt{v}}

\newcommand*{\norm}[1]{\left\lVert#1\right\rVert}
\newcommand*{\abs}[1]{\left|#1\right|}

\DeclareMathOperator*{\argmin}{argmin}

\DeclareMathOperator{\diag}{diag}

\def\cu#1{{\color{black}#1}}  
\def\pd#1{{\color{black}#1}} 
\def\gav#1{{\color{black}#1}}
\def\at#1{{\color{black}#1}}
\def\ak#1{{\color{black}#1}} 

\begin{document}

\mainmatter

\title{\pd{Gradient Methods for Problems with Inexact Model of the Objective}}

\titlerunning{Gradient Method for Problems with Inexact Model}
\author{Fedor Stonyakin\inst{1}
\and Darina Dvinskikh\inst{2,3}
\and Pavel Dvurechensky\inst{2,3}
\and Alexey Kroshnin\inst{3,4}
\and Olesya Kuznetsova\inst{4}
\and Artem Agafonov\inst{4}
\and Alexander Gasnikov\inst{3, 4, 5}
\and Alexander Tyurin\inst{5}
\and C\'{e}sar A. Uribe\inst{6}
\and Dmitry Pasechnyuk\inst{7}
\and Sergei Artamonov\inst{5}
}
\authorrunning{F. Stonyakin et al.}
\institute{
    V.\,I.\,Vernadsky Crimean Federal University, Simferopol\\
    \email{fedyor@mail.ru}
    \and Weierstrass Institute for Applied Analysis and Stochastics
    \email{darina.dvinskikh@wias-berlin.de, pavel.dvurechensky@wias-berlin.de}
    \and Institute for Information Transmission Problems RAS, Moscow
    \email{kroshnin@phystech.edu}
    \and
    Moscow Institute of Physics and Technologies, Moscow\\
    \email{gasnikov@yandex.ru, fillifyonk@gmail.com, agafonov.ad@phystech.edu}
    \and  National Research University Higher School of Economics\\
    \email{alexandertiurin@gmail.com, sartamonov@hse.ru }
    \and Massachusetts Institute of Technology, Cambridge \\
    \email{cauribe@mit.edu}\\
    \and 239-th school of St. Petersburg
     \email{pasechnyuk2004@gmail.com}
}

\maketitle

\begin{abstract}
We consider optimization methods for convex minimization problems under inexact information on the objective function. We introduce inexact model of the objective, which as a particular cases includes inexact oracle \cite{devolder2014first} and relative smoothness condition \cite{lu2018relatively}. We analyze gradient method which uses this inexact model and obtain convergence rates for convex and strongly convex problems. To show potential applications of our general framework we consider three particular problems. The first one is clustering by electorial model introduced in \cite{nesterov2018clustering}. The second one is approximating optimal transport distance, for which we propose a Proximal Sinkhorn algorithm. The third one is devoted to approximating optimal transport barycenter and we propose a Proximal Iterative Bregman Projections algorithm. We also illustrate the practical performance of our algorithms by numerical experiments.
\end{abstract}

\section{Introduction}
\pd{In this paper we consider optimization methods for convex problems under inexact information on the objective function. This information is given by an object, which we call \emph{inexact model}. Inexact model generalizes the inexact oracle introduced in \cite{devolder2014first}, where inexactness is assumed to be present in the objective value and its gradient. The authors show that, based on these two objects, it is possible to construct a linear function, which is a lower approximation and, up to a quadratic term, an upper approximation of the objective, and these two approximations are enough to obtain convergence rates for gradient method and accelerated gradient method. We go beyond and assume that the approximations of the objective are given through some function, which is not necessarily linear.} 

\pd{This allows us to construct general gradient-type method which is applicable in for different problem classes and allows to obtain convergence rates in these situations as a corollary of our general theorem. Besides convex problems we focus also on strongly convex objectives and illustrate the application of our general theory by two examples. The first example is data clustering by electoral model \cite{nesterov2018clustering}. The second example relates to Wasserstein distance and barycenter, which are widely used in data analysis \cite{cuturi2013sinkhorn,cuturi2014fast}.}

\pd{
Many optimization methods use some model of the objective function to define a step by minimization of this model. Usually the model is constructed using exact first-order \cite{nesterov2004introduction,drusvyatskiy2016nonsmooth,ochs2017non}, second-order \cite{nesterov2006cubic}, or higher-order information \cite{cartis2017improved,nesterov2018implementable} information on the objective.
The influence of inexactness on the convergence of gradient-type methods have being studied at least since \cite{polyak1987introduction}. Accelerated first-order methods with inexact oracle are studied in \cite{aspremont2008smooth,mairal2013optimization,devolder2014first,dvurechensky2016stochastic,cohen2018acceleration}. Some recent works study also non-convex problems in this context \cite{bogolubsky2016learning,dvurechensky2017gradient}. Randomized methods with inexact oracle are also studied in the literature, e.g. coordinate descent in \cite{tappenden2016inexact,dvurechensky2017randomized}, random gradient-free methods and random directional derivative methods in \cite{dvurechensky2018accelerated,dvurechensky2018accelerated2}. A method with inexact oracle for variational inequalities can be found in \cite{dvurechensky2018generalized}. }



The contributions of this paper \cu{can be summarized} as follows.

\pd{
\begin{itemize}
    \item[$\square$] We introduce an inexact model of the objective function for convex optimization problems and strongly convex optimization problems.
    \item[$\square$] We introduce and theoretically analyze a gradient-type method for convex and strongly convex problems with an inexact model of the objective function. For the latter case we prove linear rate of convergence.
    \item[$\square$] We apply our method to, generally speaking, non-convex optimization problem which arises in clustering model introduced in \cite{nesterov2018clustering}. To do this we construct an inexact model and apply our general algorithms and convergence theorems.
    \item[$\square$] We apply our general framework for Wasserstein distance and barycenter problems and show that it allows to construct a proximal \'a la \cite{chen1993convergence} version of the Sinkhorn's algorithm \cite{sinkhorn1974diagonal} and Iterative Bregman Projection algorithm \cite{benamou2015iterative}.
\end{itemize}
}

\pd{
\textbf{Notation.}
We define $\textbf{1} = (1,...,1)^T \in \R^n$, $KL(z | t)$ to be the Kullback-Leibler divergence:
$KL(z|t) = \sum\limits_{k=1}^{n} z_k \ln (z_k/t_k)$,  $\forall z, t \in S_n(1)$,
where $S_n(1)$ is the standard simplex in $\R^n$. We also denote by $\odot$ the entrywise product of two matrices.
}

\section{Gradient Methods with Inexact Model of the Objective}
\label{S:Inexact_model}
Consider \cu{the} convex optimization problem 
\begin{gather}\label{Problem}
f(x) \rightarrow \min_{x \in Q}\cu{,}
\end{gather}
where function \cu{$f$} is convex and $Q \subseteq \R^n$ is a simple convex compact 
set.
\cu{Moreover,}
assume that $\min_{x \in Q} f(x) = f(x_*)$ for some $x_* \in Q$.  


\pd{To solve this problem, we introduce a norm $\|\cdot\|$ on $\R^n$ and a prox-function $d(x)$ which is continuous and convex. We underline that, unlike most of the literature, we do not require $d$ to be strongly convex.
Without loss of generality, we assume that $\min\limits_{x\in \R^n} d(x) = 0$. 
Further, we define \textit{Bregman divergence} $V[y](x) := d(x) - d(y) - \la \nabla d(y), x - y \ra$. Next we define the inexact model of the objective function, which generalizes the inexact oracle of \cite{devolder2014first} (see also \cite{dvurechensky2016stochastic,bogolubsky2016learning,dvurechensky2017universal,gasnikov2017universal,stonyakin2019inexact,tyurin2017fast}).}


\begin{definition}	
\label{model} 
	Let function $\psi_{\delta}(x, y)$ be convex in $x \in Q$ and satisfy $\psi_{\delta}(x, x) = 0$ for all $x \in Q$.
	
	i) We say that $\psi_{\delta}(x, y)$ is \cu{a} $(\delta, L)$-model of the function $f$ at a given point $y$ with respect to $V[y](x)$ iff, 
	for all $x \in Q$, the inequality
	\begin{gather}
	\label{model_def}
	0 \le f(x) - (f(y) + \psi_{\delta}(x, y)) \le LV[y](x) + \delta
	\end{gather}
	holds for some $L, \delta > 0$.
	
	ii) We say that $\psi_{\delta}(x, y)$ is \cu{a} $(\delta, L,\mu)$-model of the function $f$ at a given point $y$ with respect to $V[y](x)$ iff, 
	for all $x \in Q$, the inequality
\begin{equation}\label{Ineq_Bregman_Model}
\mu V[y](x) \le f(x) - (f(y) + \psi_{\delta}(x, y)) \le LV[y](x) + \delta
\end{equation}
\end{definition}

\pd{Note that we allow $L$ to depend on $\delta$. We refer to the case i) as convex case and to the case ii) as strongly convex case.} 
\pd{
\begin{remark}
In the particular case of function $f$ possessing  $(\delta, L)$-oracle \cite{devolder2014first} at a given point $y$, one has  
$$
0 \leq f(x) - f(y) - \langle g_{\delta}(y), x - y\rangle \leq \frac{L}{2}\norm{x - y}^2 + \delta
$$
and $\psi_{\delta}(x, y) = \langle g_{\delta}(y), x - y\rangle$.
In the same way, if function $f$ is equipped with $(\delta, L,\mu)$-oracle \cite{devolder2013firstCORE}, i.e.,
$$
 \frac{\mu}{2}\norm{x - y}^2 \leq f(x) - f(y) - \langle g_{\delta, L,\mu}(y), x - y\rangle \leq \frac{L}{2}\norm{x - y}^2 + \delta \quad \forall x \in Q,
$$
we have $\psi_{\delta}(x, y) = \langle g_{\delta, L,\mu}(y), x - y\rangle$.
\end{remark} 
}

\pd{The algorithms we develop are based on solving auxiliary simple problems on each iteration. We assume that these problems can be solved inexactly and, following  \cite{ben-tal2015lectures} introduce a definition of inexact solution of a problem.}

\begin{definition}
\label{def_precision}
\label{solNemirovskiy}
Consider a convex minimization problem 
\begin{align}\label{min_prob}
\phi(x) \rightarrow \min_{x \in Q \subseteq \R^n}.
\end{align}
If $\phi$ is smooth, we say that we solve it with $\widetilde{\delta}$-`precision' ($\widetilde{\delta} \geq 0$) if we find $\tilde{x}$ s.t.
$ \max_{x\in Q}\langle \nabla \phi(\tilde x), \tilde{x}-x \rangle = \widetilde{\delta}$.
If $\phi$ is general convex, we say that we solve this problem with $\widetilde{\delta}$-`precision' if we find $\tilde{x}$ s.t. 
$\exists h \in \partial\gav{\phi}(\widetilde{x}), \,\,\, \langle h, x_* - \widetilde{x} \rangle \geq -\widetilde{\delta}$.
In both cases we denote this $\tilde{x}$ as $\argmin_{x \in Q}^{\widetilde{\delta}}\gav{\phi}(x)$.
\end{definition}
We notice that the case $\widetilde{\delta} = 0$ corresponds to the case when $\tilde x$ is an exact solution of convex optimization problem~\eqref{min_prob} \cite{ben-tal2015lectures,nesterov2004introduction}.
The connection of Definition \ref{solNemirovskiy} with standard definitions of inexact solution, e.g. in terms of the objective residual, can be found in Appendix~\ref{InexactSolutions}.

\subsection{Convex Case}
\pd{In this subsection we describe a gradient-type method for problems with  $(\delta, L)$-model of the objective. This algorithm is a natural extension of gradient method, see \cite{gasnikov2017universal,stonyakin2019inexact,tyurin2017fast}.}

\begin{algorithm}
\caption{\pd{Gradient method with $(\delta, L)$-model of the objective.}}
\label{Alg1_nonadaptive}
\begin{algorithmic}[1]
\STATE \textbf{Input:} $x_0$ is the starting point, $L > 0$ and 
$\delta,\widetilde{\delta}>0$.
\FOR{$k \geq 0$}
\STATE \begin{equation}
\phi_{k+1}(x) := \psi_{\delta}(x, x_k)+L V[x_k](x), \quad
x_{k+1} := {\arg\min_{x \in Q}}^{\widetilde{\delta}} \phi_{k+1}(x). 
\end{equation}
\ENDFOR
\ENSURE 
$\bar{x}_N= \frac{1}{N}\sum_{k=0}^{N-1}x_{k+1}$
\end{algorithmic}
\end{algorithm}



\begin{theorem}
	\label{mainTheoremDL_G}
		Let $V[x_0](x_*) \leq R^2$, where $x_0$~ is the starting point, and $x_*$~ is the nearest minimum point to the point $x_0$ in the sense of Bregman divergence $V[y](x)$. Then, for the sequence, generated by Algorithm~\ref{Alg1_nonadaptive} the following inequality holds:
	\begin{equation}
	f(\bar{x}_N) - f(x_*) \leq  \frac{LR^2}{N} + \widetilde{\delta} + \delta,
	\end{equation}
\end{theorem}
In appendix \ref{ProofGM} we prove this theorem and provide an adaptive version of Algorithm \ref{mainTheoremDL_G}, which does not require knowledge of the constant $L$.

\subsection{Strongly Convex Case}\label{SC}
\pd{
In this subsection we consider problem \eqref{Problem} with $(\delta, L, \mu)$-model of the objective function satisfying \eqref{Ineq_Bregman_Model}. This more strong assumption allows us to obtain linear rate of convergence of the proposed algorithm.
Our algorithm is listed as Algorithm \ref{Alg2} and it is a version of Algorithm  \ref{Alg1_nonadaptive}, which is adaptive to possibly unknown constant $L$.
}






\begin{algorithm}
\caption{Adaptive gradient method with an oracle using the $(\delta, L, 
\mu)$-model}
\label{Alg2}
\begin{algorithmic}[1]
\STATE \textbf{Input:} $x_0$ is the starting point, $\mu>0$ $L_{0} \geq 2\mu$ and 
$\delta$.
\STATE Set
 $S_0 := 0 $
\FOR{$k \geq 0$}
\STATE Find the smallest $i_k\geq 0$ such that
\begin{equation}\label{exitLDL_G_S}
f(x_{k+1}) \leq f(x_{k}) + \psi_{\delta}(x_{k+1}, x_{k}) +L_{k+1}V[x_{k}](x_{k+1}) + \delta,
\end{equation}
where $L_{k+1} = 2^{i_k-1}L_k$ for $L_k \geq 2 \mu$ and $L_{k+1} = 2^{i_k}L_k$ for $L_k < 2 \mu$,\\ $\alpha_{k+1}:= \frac{1}{L_{k+1}}$, $S_{k+1} := S_k + \alpha_{k+1}$.
\begin{equation}\label{equmir2DL_G_S}
\phi_{k+1}(x) := \psi_{\delta}(x, x_k)+L_{k+1}V[x_k](x), \quad
x_{k+1} := {\arg\min_{x \in Q}}^{\widetilde{\delta}} \phi_{k+1}(x).
\end{equation}
\ENDFOR
\ENSURE 
$\bar{x}_N= \frac{1}{S_N}\sum_{k=0}^{N-1}\frac{x_{k+1}}{L_{k+1}}$
\end{algorithmic}
\end{algorithm}

Let's introduce average parameter $\hat{L}$:
	$$1-\dfrac{\mu}{\hat{L}} = \sqrt[k+1]{ \left(1-\dfrac{\mu}{L^{k+1}}\right)\left(1-\dfrac{\mu}{L_{k}}\right)\ldots\left(1-\dfrac{\mu}{L_{1}}\right)}.
	$$
Note that by $L_i \ge \mu~(i=1,2,\ldots)$ 
	$$\min\limits_{1\le i \le k+1}L_i \le \hat{L} \le \max\limits_{1\le i \le k+1}L_i \pd{\leq 2L}.$$
 
The following result hold\pd{s}.
\begin{theorem}\label{thm_fedyor}
Let $\psi_{\delta}(x, y)$ is a $(\delta, L, \mu)$-model for $f$ w.r.t. $V[y](x)$. Then, after $k$ iterations of Algorithm \ref{Alg2}, we have 
\begin{equation}\label{fedyor_strong_1}
V[x^{k+1}](x_*) \le \dfrac{2L(\delta+\widetilde{\delta})}{\mu^2}\left(1 - \left(1-\dfrac{\mu}{2L}\right)^{k+1}\right) + \left(1-\dfrac{\mu}{\hat{L}}\right)^{k+1}V[x^0](x_*),
\end{equation}
\begin{equation}\label{fedyor_strong_2}
f(x^{k+1}) - f(x_*) \le  \dfrac{4L^2(\delta+\widetilde{\delta})}{\mu^2}\left(1 - \left(1-\dfrac{\mu}{2L}\right)^{k+1}\right) + 2L\left(1-\dfrac{\mu}{\hat{L}}\right)^{k+1}V[x^0](x_*).
\end{equation}	
\end{theorem}

The details of proof can be found in Appendix \ref{sec: ProofThFedyor}. \pd{Note that Algorithm \ref{Alg1_nonadaptive} also has linear convergence rate for the strongly convex case. The details can be found in Appendix \ref{sec: ProofTh}. The benefit of Algorithm \ref{Alg1_nonadaptive} is that there is no need to know the strong convexity parameter $\mu$ for the algorithm to work. On the other hand, this parameter is needed for assessing the quality of the solution returned by the algorithm. The benefit of the adaptive version is that it does not require to know the value of the parameter $L$ and adapts to it. Moreover, the parameter $L$ can be different for the model at different points and the algorithm adapts also for the local value of this parameter.}

\section{\pd{Clustering by Electorial Model}}
\label{S:clustering}
\pd{In this section we consider clustering model introduced in \cite{nesterov2018clustering}. In this model voters (data points) choose a party (cluster) in an iterative manner by alternative minimization of the following function.}
\begin{equation}\label{Nesterov_Electoral_Model}
f_{\mu_1, \mu_2}(x = (z, p)) = g(x) + \mu_1 \sum\limits_{k=1}^{n}z_k \ln z_k + \frac{\mu_2}{2} \|p\|^2_2 \rightarrow \min_{z \in S_n(1), p \in \R^m_{+}},
\end{equation}
\pd{where $\R_{+}^m$ is a non-negative orthant and $S_n(1)$ is the standard $n$-dimensional simplex in $\R^n$.}
\pd{The vector $z$ contains probabilities with which voters choose the considered party, and vector $p$ describes the position of the party in the space of voter opinions. The minimized potential is the result of combining two optimization problems into one: voters choose the party whose position is closest to their personal opinion and the party adjusts its position minimizing dispersion and trying not to go too far from its initial position. Yu. Nesterov in \cite{nesterov2018clustering} used sequential elections process to show that under some natural assumptions the process convergence and gives the clustering of the data-points. This was done for a particular choice of the function $g$ which has limited interpretability. We show, how our framework of inexact model of the objective allows to construct a gradient-type method for the case of general function $g$, which is not necessarily convex.}


Assume that $g(x)$ (generally, non-convex) is an function with $L_g$-Lipschitz continuous gradient:
\begin{equation}\label{lipcon}
	\|\nabla g(x) - \nabla g(y)\|_{*} \le L_g\|x - y\| \quad \forall x, y \in S_n(1) \times \R_{+}^m,
	\end{equation}
\pd{and, following \cite{nesterov2018clustering}, the numbers $\mu_1, \mu_2$ are chosen such that} $L_g \le \mu_1$ and $L_g \le \mu_2$.
	
The norm $\|\cdot\|$ in $S_n(1)\times \R_{+}^m$ is defined as 
$\|(z, p)\|^2 = \|z\|^2_1 + \|p\|^2_2$,
where $\|z\|_1 = \sum\limits_{k = 1}^n z_k$ and $\|p\|_2 = \sqrt{\sum\limits_{k = 1}^m p_k^2}$. This is indeed a norm since, for $x = (z_x, p_x)$ and $y = (z_y, p_y)$ we have: 

$$\|x + y\| = \sqrt{\|z_x + z_y\|^2_1 + \|p_x + p_y\|^2_2} \le \sqrt{(\|z_x\|_1 + \|z_y\|_1)^2 + (\|p_x\|_2 + \|p_y\|_2)^2} \le $$
$$\le \sqrt{\|z_x\|^2_1 + \|p_x\|^2_2} + \sqrt{\|z_y\|^2_1 + \|p_y\|^2_2} = \|x\| + \|y\|,$$
because $\sqrt{(a+b)^2 + (c+d)^2} \le \sqrt{a^2 +c^2} + \sqrt{b^2 +d^2} $ for each $a, b, c, d \ge 0$.

Let us show that
$$
\psi_{\delta}(x, y) = \langle \nabla g(y), x-y \rangle  - L_g \cdot KL(z_x|z_y) - \frac{L_g }{2}\|p_x-p_y\|^2_2  +
$$
$$
+ \mu_1 (KL(z_x | \textbf{1}) - KL(z_y | \textbf{1})) + \frac{\mu_2 }{2} \left(\|p_x\|^2_2-\|p_y\|^2_2\right)
$$
is a $(0, 2L_g)$-model of $f_{\mu_1, \mu_2}(x)$ in $x$ with respect to the following Bregman \pd{divergence}
\begin{eqnarray*}
V[y](x) = KL(z_x | z_y) + \frac{1}{2}\|p_x - p_y\|^2_2.
\end{eqnarray*}
 
It is easy to see that $\psi_{\delta}(x, x) = 0$. 
Let us show, that inequality \eqref{model_def} holds for $\psi_{\delta}(x,y)$. For the function $g(x)$ satisfying \eqref{lipcon} we have:
\begin{equation}\label{ineq1}
|g(x) - g(y) - \langle \nabla g(y), x-y\rangle | \leq \frac{L_g}{2}\|x - y\|^2.
\end{equation}

It means that $f_{\mu_1, \mu_2}(x) - f_{\mu_1, \mu_2}(y) - \psi_{\delta}(x, y) =$
\begin{gather}
	 =  g(x) - g(y)  - \langle \nabla g(y), x-y\rangle + \mu_1 \cdot KL(z_x | \textbf{1}) - \mu_1 \cdot KL(z_y | \textbf{1}) +\\+ \frac{\mu_2}{2}\|p_x\|^2_2 - \frac{\mu_2}{2}\|p_y\|^2_2 - \mu_1 \cdot KL(z_x | \textbf{1}) + \mu_1 \cdot KL(z_y | \textbf{1}) + L_g \cdot KL(z_x | z_y) + \\ + \frac{L_g}{2}\|p_x - p_y\|^2_2 -\frac{\mu_2}{2}\|p_x\|^2_2 + \frac{\mu_2}{2}\|p_y\|^2_2 = 
	\end{gather}
	$$
	= g(x) - g(y) - \langle \nabla g(y), x-y\rangle + L_g \cdot KL(z_x | z_y) + \frac{L_g}{2}\|p_x - p_y\|^2_2.
	$$
	Along with \eqref{ineq1} and $KL(z_x | z_y) \ge \frac{\|z_x - z_y\|_1^2}{2}$ it \pd{leads} to
	\begin{eqnarray*}
	f_{\mu_1, \mu_2}(x) - f_{\mu_1, \mu_2}(y) - \psi_{\delta}(x, y) \le \frac{L_g}{2}\|x - y\|^2 + L_g\cdot KL(z_x | z_y) + \frac{L_g}{2}\|p_x - p_y\|^2_2, \\
		f_{\mu_1, \mu_2}(x) - f_{\mu_1, \mu_2}(y) - \psi_{\delta}(x, y) \ge -\frac{L_g}{2}\|x - y\|^2 + L_g \cdot KL(z_x | z_y) + \frac{L_g}{2}\|p_x - p_y\|^2_2.
	\end{eqnarray*}

Finally, by definition of the norm $\|\cdot\|$\pd{,} we have
	\begin{eqnarray*}
	0 \le f_{\mu_1, \mu_2}(x) - f_{\mu_1, \mu_2}(y) - \psi_{\delta}(x, y) \le 2L_g \cdot KL(z_x | z_y) + L_g\|p_x - p_y\|^2_2 = 2L_g V[y](x),
	\end{eqnarray*}
i.e. $\psi_{\delta}(x, y)$ is a $(0, 2L_g)$-model of the function $f_{\mu_1, \mu_2}$. 

Further, for the case $\min\{\mu_1, \mu_2\} > L_g$ $\psi_{\delta}(x, y)$ is a strongly convex w.r.t. $V[y](x)$:
\begin{equation}
\psi_{\delta}(x, y) = \psi_{\delta}^{lin}(x, y) + (\mu_1 - L_g) \cdot KL(z_x |z_y) + \frac{\mu_2 - L_g}{2} \|p_x-p_y\|^2_2 \ge
\end{equation}
$$
\ge (\min\{\mu_1, \mu_2\} - L_g) \cdot V[y](x),
$$
where
\begin{equation}
\psi_{\delta}^{lin}(x, y) = \langle \nabla g(y), x-y \rangle + \mu_1 \langle \nabla KL(z_y | 1), z_x - z_y\rangle + \mu_2 \langle p_y , p_x - p_y\rangle
\end{equation}
is linear \pd{in} $y$. Indeed,
$$
\psi_{\delta}(x, y) = \langle \nabla g(y), x-y \rangle + \mu_1 \langle \nabla KL(z_y | 1), z_x - z_y\rangle  + \mu_2 \langle p_y , p_x - p_y\rangle -
$$
$$
- L_g \cdot KL(z_x|z_y) - \frac{L_g}{2} \|p_x - p_y\|^2_2  + \mu_1 \left(\cdot KL(z_x | \textbf{1}) - KL(z_y | \textbf{1}\right) - \langle \nabla KL(z_y | 1), z_x - z_y\rangle) + 
$$
$$
+ \frac{\mu_2 }{2}\left(\|p_x\|^2_2 - \| p_y\|^2_2 -  \langle 2\cdot p_y , p_x \rangle - p_y\right) = 
$$
$$ = \psi_{\delta}^{lin}(x, y) + (\mu_1 - L_g)\cdot KL(z_x | z_y) + \frac{\mu_2 - L_g}{2} \cdot\|p_x-p_y\|^2_2 .
$$

Thus, $\psi_{\delta}^{lin}(x, y)$ is a $(0, \max\{\mu_1, \mu_2\}+L_g, \min\{\mu_1, \mu_2\} - L_g)$-model of the function $f_{\mu_1, \mu_2}$:
\begin{equation}
  f_{\mu_1, \mu_2}(y) + \psi_{\delta}^{lin}(x, y) + (\min\{\mu_1, \mu_2\} - L_g)V[y](x) \le   f_{\mu_1, \mu_2}(x) 
\end{equation}  
and
\begin{equation}
 f_{\mu_1, \mu_2}(x) \le f_{\mu_1, \mu_2}(y) + \psi_{\delta}^{lin}(x, y) + (\max\{\mu_1, \mu_2\} + L_g)V[y](x).
\end{equation}

So, we can apply our Algorithms \ref{Alg1_nonadaptive} and \ref{Alg2} to \pd{the problem} \eqref{Nesterov_Electoral_Model}.


\section{Proximal Sinkhorn Algorithm for Optimal Transport}
\label{OT_PS}

In this section we consider the problem of approximating an optimal transport (OT) distance. Recently optimal transport distances has gained a lot of interest in machine learning and statistical applications \cite{arjovsky2017wasserstein,bigot2012consistent,barrio2015statistical,ebert2017construction,le2017existence,peyre2017computational,solomon2014wasserstein}. To state the OT problem, assume that we are given two discrete probability measures $p,q \in S_n(1)$ and  ground cost matrix $C \in \R_+^{n\times n}$, then the optimal transport problem is
\begin{align}\label{wass_dist}
     \la C,  \pi\ra
    \rightarrow 
    \min_{\pi \in \mathcal{U}(p,q)} , \;\; \mathcal{U}(p,q) = \{\pi \in \R_+^{n \times n}: \pi \boldsymbol{1} = p, \pi^T \boldsymbol{1} = q\}
\end{align}
where \pd{$\la \cdot, \cdot \ra$ denotes Frobenius dot product of matrices}, $\pi$ is a transportation plan.
The above optimal transport problem is the Kantorovich \cite{kantorovich1942translocation} linear program (LP)  formulation of the problem, which goes back to the Monge's problem \cite{monge1781memoire}. 
The best known theoretical complexity for this linear program is \footnote{Here and below for all (large) $n$: $\widetilde{O}(g(n)) \le \tilde{C}\cdot(\ln n)^r g(n)$ with some constants $\tilde{C} > 0$ and $r \ge 0$. Typically, $r = 1$, but not in this particular case. If $r=0$, then $\widetilde{O}(\cdot) = O(\cdot)$.} $\widetilde{O}(n^{2.5})$, see   \cite{lee2014path}. However, there is no known practical implementation of this algorithm. In practice, the simplex method gives complexity $O(n^3 \ln n)$ \cite{pele2009fast}. We follow the alternative approach based on entropic regularization of the OT problem \cite{cuturi2013sinkhorn}. We show how our general framework of inexact model of the objective allows to construct Proximal Sinkhorn algorithm with better computational stability in comparison with the standard Sinkhorn algorithm.  
 
For any optimization problem \eqref{Problem}, $\psi_{\delta}(x,y) = f(x) - f(y)$ satisfies Definition \ref{model} with any $L\geq 0$. In this case, our Algorithm~\ref{Alg1_nonadaptive} becomes inexact \textit{Bregman proximal gradient method}
\begin{align}\label{prox_meth}
    x^{k+1} = {\arg\min_{x\in Q}}^{\tilde{\delta}} \{f(x) + L V[x^k](x)\}. 
\end{align}

Our idea is to apply this proximal method for the OT problem and approximately find the next iterate $x^{k+1}$ by Sinkhorn's algorithm \cite{sinkhorn1974diagonal,cuturi2013sinkhorn,altschuler2017near-linear,dvurechensky2018computational}. The latter is made possible by the choice of $V$ as KL divergence, which makes the problem of finding the point $x^{k+1}$ to be an entropy-regularized OT problem, which, in turn, is efficiently solvable by the Sinkhorn algorithm.






Consider the iterates 
\begin{align}\label{prox_for_sinkhorn}
    \pi^0 = \ak{p q^T \in \mathcal{U}(p,q)}, \quad   \pi^{k+1} & = {\arg\min\limits_{\pi \in \mathcal{U}(p,q)}}^{\e/2}  \left\{
             \la C, \pi \ra + L \cdot KL(\pi| \pi^k)
        \right\}  \notag \\
        &=  {\arg\min\limits_{\pi \in \mathcal{U}(p,q)}}^{\e/2} KL\left(\pi \left| \pi^k \odot \exp\left(-\frac{C}{L}\right)\right.\right),  
\end{align}
which we call outer iterations. On each outer iteration we use Sinkhorn's algorithm \ref{Alg:Sinkhorn}, which solves the minimization problem in \eqref{prox_for_sinkhorn} with accuracy $\tilde{\e}$ in terms of its objective residual. 
Notice that unlike \cite{dvurechensky2018computational} we provide a slightly refined theoretical bounds for the Sinkhorn's algorithm not depending on vectors $p$, $q$.

\begin{algorithm}[tb]
\caption{Sinkhorn's Algorithm}
\label{Alg:Sinkhorn}
  \begin{algorithmic}[1]
     \REQUIRE Accuracy $\tilde{\e}$, matrix $K = e^{- C / \gamma}$, marginals $p, q \in S_n(1)$.
     \STATE Set $t = 0$, \ak{$u^0 = \ln p$, $v^0 = \ln q$, $\e' = \frac{\tilde{\e}}{4} \left(\max_{i, j} C_{i j} - \min_{i, j} C_{i j} + 2 \gamma \ln\left(\tfrac{4 \gamma n^2}{\tilde{\e}}\right)\right)^{-1}$}.
     \REPEAT
        \IF{$t \bmod 2 = 0$}
        	\STATE $u^{t + 1} = u^t + \ln p - \ln(B(u^t, v^t) \one)$, where $B(u, v) := \diag(e^u) K \diag(e^v)$
            \STATE $v^{t + 1} = v^t$
        \ELSE
        	\STATE $v^{t + 1} = v^t + \ln q - \ln(B(u^t, v^t)^T \one)$
            \STATE $u^{t + 1} = u^t$
        \ENDIF
        \STATE $t = t + 1$
	\UNTIL{$\norm{B(u^t, v^t) \one - p}_1 + \norm{B(u^t, v^t)^T \one - q}_1 \le \e'$}
	\STATE Find $\hat{\pi}$ as the projection of $B(u^t, v^t)$ on $\U(p,q)$ by Algorithm~2 in \cite{altschuler2017near-linear}.
    \ENSURE $\hat{\pi}$.	
  \end{algorithmic}
\end{algorithm}

\begin{theorem}\label{Th:ProxSinkhorn}
Let $\bar \pi^N = \frac{1}{N} \sum_{k=1}^N \pi^k$, where $\pi^k$ are the iterates of \eqref{prox_for_sinkhorn}. Then, after $N = \frac{4 L \ln n}{\e}$ iterations, it holds that $\la C, \bar \pi^N \ra \leq 
    \min_{\pi \in \mathcal{U} (p,q)} \la C, \pi \ra + \e$. Moreover, the accuracy $\tilde{\e}$ for the solution of \eqref{prox_for_sinkhorn}  is sufficient to be set as $\widetilde{O}(\ak{\e^4 / (L n^4)})$ and the complexity of Sinkhorn's Algorithm on $k$-th iteration is bounded as
\begin{equation}\label{IterationCost}
    \ak{n^2 \widetilde{O}\left(\min\Bigg\{\exp\left(\frac{\bar{c}_k}{L}\right) \left(\frac{\bar{c}_k}{L} + \ln\frac{\bar{c}_k}{\tilde{\e}}\right),\, \frac{\bar{c}_k^2}{L \tilde{\e}}\Bigg\}\right)},
\end{equation}
where\footnote{This bound is rough and typically $\bar{c}_k$ is smaller in practice. By proper rounding of $\pi^k$ one can guarantee (without loss of generality) that $\pi^k_{ij}\ge \e/(2 n^2 \ak{\norm{C}_\infty})$, which gives 
\[
\ak{\frac{\bar{c}_k}{L} = \frac{\norm{C}_\infty}{L} + \ln\left(\frac{2 n^2 \norm{C}_\infty}{\e}\right).}
\]
But, in practice there often is no need to make `rounding' after each outer iteration. 
}  
\begin{equation}\label{Cbound}
    \ak{\bar{c}_k = \norm{C}_\infty + L \ln\left(\frac{\max_{i, j} \pi^k_{i j}}{\min_{i, j} \pi^k_{i j}}\right)}.
\end{equation}

\end{theorem}


\begin{proof}
The estimate for the number of iterations $N$ follows from Theorem~\ref{mainTheoremDL_G} since $V[\pi_0](\pi_*) \le \ln n^2$ as $\pi \in S_{n^2}(1)$. The first component of~\eqref{IterationCost} is proved in \cite{franklin1989scaling}, and the second component basically follows from \cite{beck2015convergence,dvurechensky2018computational}. Proofs of the second component and bound on $\bar{c}_k$~\eqref{Cbound} are provided in Appendix~\ref{sec:Sinkhorn} (Theorem~\ref{thm:Sinkhorn_complexity}). 
Let us show that it is sufficient to solve minimization problem  \eqref{prox_for_sinkhorn} on each iteration with accuracy $\tilde{\e} = \widetilde O(\ak{\e^4 / (L n^4)})$ in terms of the objective residual to guarantee $\tilde{\delta} = \e / 2$ accuracy in terms of Definition~\ref{def_precision}.

To prove this fact, we use relation \eqref{inexact} in Theorem~\ref{lm:delta_eps} of Appendix~\ref{InexactSolutions} with $\|\cdot\| = \|\cdot\|_1$, $\tilde{R} = 2$, $\mu = L$. To bound $\Delta = \tilde{L}\widetilde{R} + \|\nabla\phi(\tilde{x}^*)\|_*$ (in notations of Theorem~\ref{lm:delta_eps}) we modify  $\mathcal{U}(p,q)$  by adding constraints: $\pi_{ij} \ge \e / (4 \ak{n^2})$, $i,j=1,...,n$. 
The solution of the changed problem is still an $O(\e)$-solution of the original problem.
For the modified problem $\Delta = 5 L \ak{n^2} \tilde{R} / \e$. According to \eqref{inexact}  one should solve auxiliary problem with accuracy by function value $\widetilde{\e}$, which is chosen such that $\e/2 = \tilde{\delta} = (5L \ak{n^2}/\e) \tilde{R} \sqrt{2\widetilde{\e}/L}$. The only problem is that now we cannot directly apply Sinkhorn's algorithm. This problem can be solved by trivial affine transformation of $\pi$-space. This transformation reduces modified polyhedral to the standard one and we can use Sinkhorn's Algorithm. Such a transformation doesn't change (in terms of $O(~)$) the requirements to the accuracy.
\end{proof}


\piccaption[]{Adaptive choice of L}
\parpic[r]{
    \vspace{-20pt}
    \centering
    \includegraphics[width=0.4\textwidth]{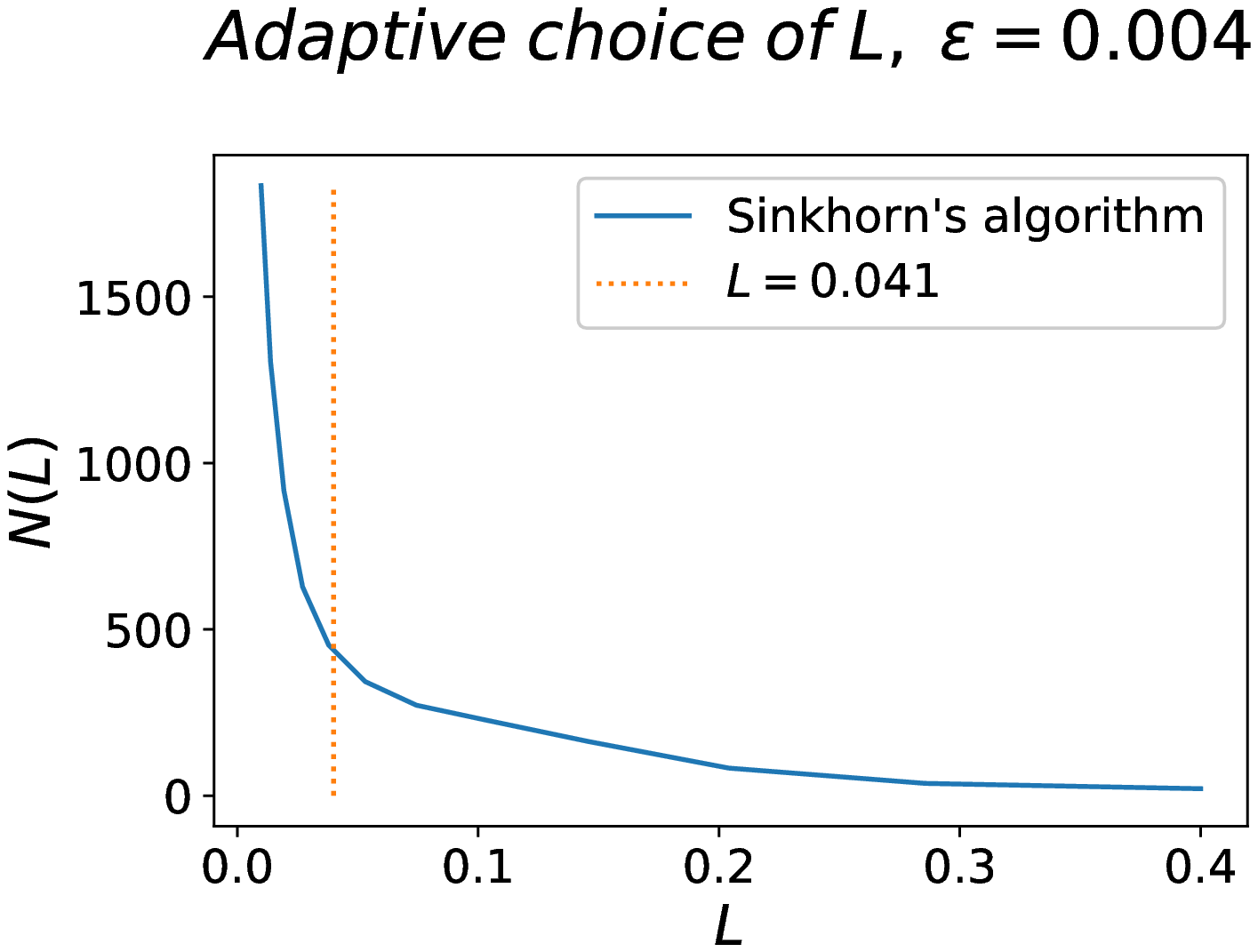}
    \vspace{-10pt}
    \label{fig:gammachoosing}
    \vspace{-10pt}
    }
\begin{remark}\label{ChoiceL}
The standard Sinkhorn's method can be seen as a particular case of our algorithm \eqref{prox_for_sinkhorn} with only one step. To obtain an $\e$-approximate solution of \eqref{wass_dist}, the regularization parameter $L$ needs to be chosen $O\left(\e / \ln n\right)$ \cite{altschuler2017near-linear,dvurechensky2018computational,gasnikov2015universal}. This can lead to instability of the Sinkhorn's algorithm \cite{schmitzer2016stabilized}. On the opposite, our Proximal Sinkhorn algorithm allows to run Sinkhorn's algorithm with larger regularization parameter. This parameter can be chosen by minimization of the theoretical bound \eqref{IterationCost}, which gives $L = \widetilde{O}(\|C\|_{\infty})$. 
In practice one can choose this constant adaptively since we have a $(\delta,L)$-model for any $L$ and can vary $L$ from iteration to iteration. First, the inner problem \eqref{prox_for_sinkhorn} is solved with overestimated $L$. Then, we set $L := L/2$ and the problem is solved with the updated value of the parameter and so on until a significant increase (e.g. 10 times) in the complexity of the auxiliary entropy-linear programming problem in comparison with the initial complexity is detected, see Figure~\hyperref[fig:gammachoosing]{1}, where $N(L)$ is a number of required iterations of Sinkhorn algorithm to solve the inner problem with accuracy $\e$. 
\end{remark}

From the Theorem~\ref{Th:ProxSinkhorn} and Remark~\ref{ChoiceL} one can roughly estimate the total complexity of Proximal Sinkhorn algorithm as\footnote{Our experiments on MNIST data set show (see Figures~\ref{fig:sinkhorn_iter},~\ref{fig:sinkhorn_time},~\ref{fig:sinkhorn_inner}) that in practice the bound is better.} 
$\ak{\widetilde{O}(n^4 / \e^2)}$.
We also mention several recent complexity bounds\footnote{Strictly speaking for the moment we can not verify all the details of the proof of estimate $\tilde{O}(n^2/\e)$. Also the proposed in \cite{blanchet2018towards,quanrud2018approximating} methods are mainly theoretical, like Lee--Sidford's method for OT problem with the complexity $\tilde{O}(n^{2.5})$ \cite{lee2014path}. For the moment it is hardly possible to implement these methods such that theirs practical efficiencies correspond to the theoretical ones.} 
for the OT problem $\tilde{O}(n^2/\e^3)$ \cite{altschuler2017near-linear}, $\tilde{O}(n^2/\e^2)$  and $\tilde{O}(n^{2.5}/\e)$ \cite{dvurechensky2018computational}, $\tilde{O}(n^2/\e)$ \cite{blanchet2018towards,quanrud2018approximating}, $\tilde{O}(n/\e^{3+d})$, $d\ge 1$ \cite{altschuler2018approximating}.

\subsection{\pd{Numerical Illustration}}
\pd{In this subsection we provide numerical illustration of the Proximal Sinkhorn algorithm.\footnote{The code is available at \url{https://github.com/dmivilensky/Proximal-Sinkhorn-algorithm}}}
In the experiments we use a standard MNIST dataset with images scaled to a size $10 \times 10$. The vectors $p$ and $q$ contain the pixel intensities of the first and second images respectively. The value of $c_{ij}$ is equal to the Euclidean distance between the $i$-th pixel from the vector $p$ and the $j$-th pixel from the vector $q$ on the image pixel grid. For experiments with varying number of pixels $n$ the images are resized to be images of $10\cdot m \times 10\cdot m$ pixels, where $m \in \mathbb{N}$. We replace all the zero elements in $p$ and $q$ with $10^{-3}$ and, then, normalize these vectors.
\begin{figure}[H]
    \centering
    \vspace{-2em}
    \includegraphics[width=0.86\textwidth]{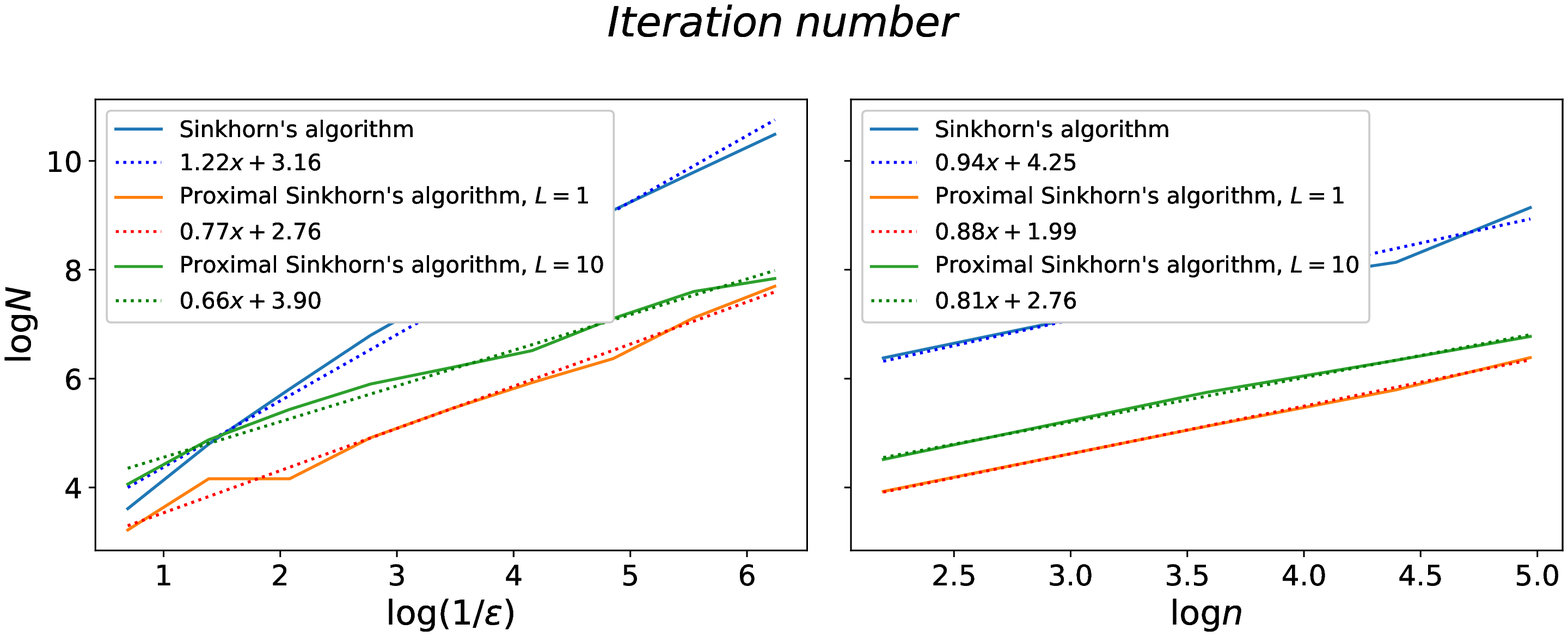}
    \caption{Comparison of iteration number of Sinkhorn's algorithm and total number of Sinkhorn steps in Proximal Sinkhorn's algorithm for different $L$. }
    \label{fig:sinkhorn_iter}
\end{figure}
Fig. \ref{fig:sinkhorn_iter} shows that the growth rate of the iteration number with increasing accuracy or size of the problem for the Sinkhorn's algorithm is greater than for the Proximal Sinkhorn's method. At the same time, with a higher value of $L$ in proximal method, the iteration number is greater, and the growth rates with some precision are equal. The same type of dependence on the accuracy and the size of the problem can be seen for the working time (fig. \ref{fig:sinkhorn_time}):

\begin{figure}[H]
    \centering
    \vspace{-2em}
    \includegraphics[width=0.86\textwidth]{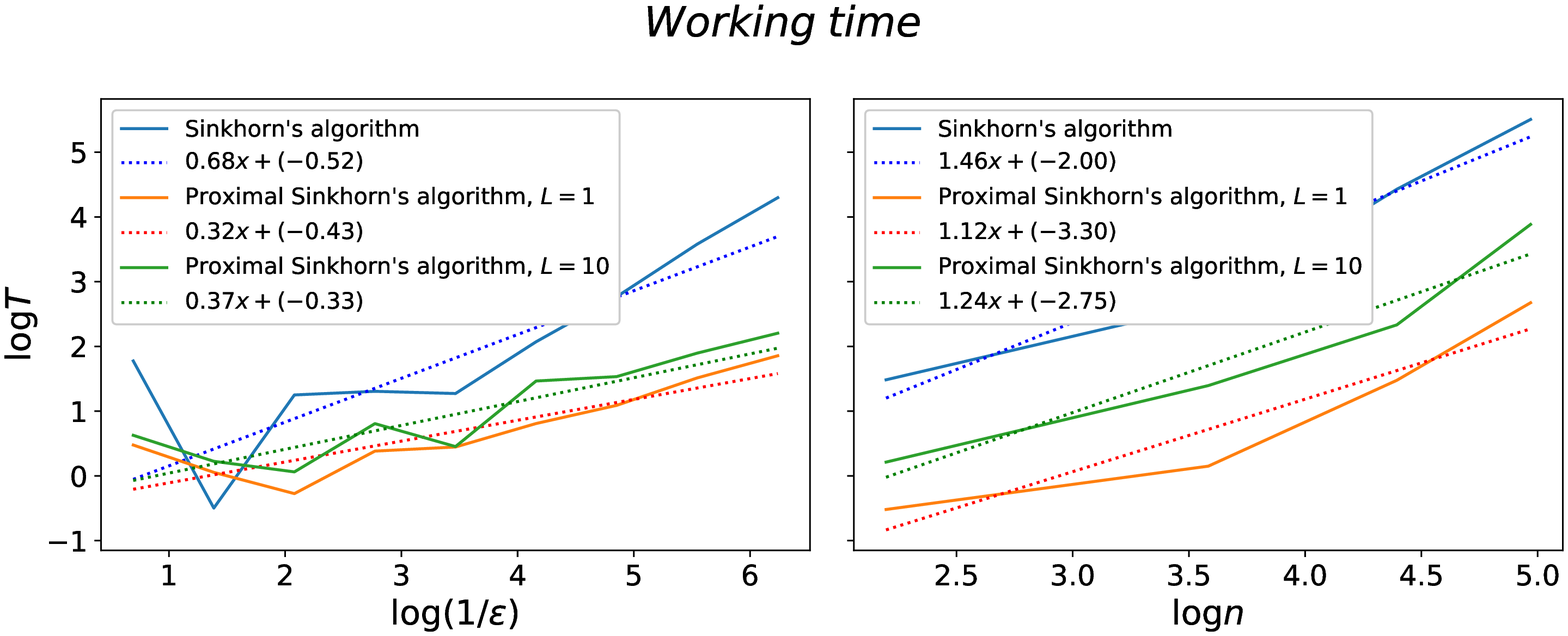}
    \caption{Comparison of working time of Sinkhorn's algorithm and Proximal Sinkhorn's algorithm with different $L$. }
    \label{fig:sinkhorn_time}
\end{figure}


\section{Proximal IBP Algorithm for Wasserstein Barycenter}
\label{sec:Prox_IBP}
In this section we consider a more complicated problem of approximating an OT barycenter. OT barycenter is a natural definition of a mean in a space endowed with an OT distance. Such barycenters are used in the analysis of data with geometric structure, e.g. images, and other machine learning applications \cite{cuturi2014fast,benamou2015iterative,kroshnin2019complexity,peyre2017computational,kroshnin2019statistical}.
For a set of probability measures $\{p_1, \dots, p_m\}$, cost matrices $C_1, \dots, C_m \in \R_+^{n \times n}$, and $w \in S_n(1)$, the weighted barycenter of these measures is defined as a solution of the following convex optimization problem 
\begin{equation}\label{prob:unreg_bary}
    \sum_{l=1}^m w_l \min_{\pi_l \in \U(p_l,q)} \la C_l,  \pi_l \ra \rightarrow \min_{q \in S_n(1)}  \Leftrightarrow \sum_{l=1}^m w_l \la C_l,  \pi_l \ra \rightarrow \min_{\Bpi \in \C_1 \cap \C_2},
\end{equation}
\begin{align}
    \C_1 = \left\{\Bpi = [\pi_1, \dots, \pi_m] : \forall l ~ \pi_l \one = p_l \right\},  \;\;\;
    \C_2 &= \left\{\Bpi = [\pi_1, \dots, \pi_m] :   \pi_1^T \one =  \dots = \pi_m^T \one \right\}.
\end{align}

The idea is similar to the one in Sect. \ref{OT_PS}, namely, we use our framework to define a Proximal Iterative Bregman Projections algorithm.
The algorithm starts from the point $\Bpi$ s.t. $\pi^0_l = \ak{\frac{1}{n} p_l \one^T \in \mathcal{U}(p_l, \one / n)}$, $l=1,...,m$ and iterates
\begin{align}\label{algo:prox_IBP}
     \Bpi^{k+1} &=  {\arg\min_{ \Bpi \in \C_1 \cap \C_2}}^{\e/2} \sum_{l=1}^m {w_l}\left\{ \la C_l, \pi_l  \ra + L \cdot KL(\pi_l | \pi_l^k)\right\}\notag\\
    &= {\arg\min_{\Bpi \in \C_1 \cap \C_2}}^{\e/2} \sum_{l=1}^m {w_l} KL\left(\pi_l \left| \pi_l^k \odot \exp\left(-\frac{C_l}{L}\right)\right.\right).
\end{align}
These iterations are called outer iterations and on each such iteration, the Iterative Bregman Projections algorithm \cite{benamou2015iterative} listed as Algorithm \ref{Alg:dual_IBP} below is used to solve the auxiliary minimization problem.

\begin{algorithm}[H]
    \caption{Iterative Bregman Projection }
    \label{Alg:dual_IBP}
    \begin{algorithmic}[1]
        \REQUIRE $C_1, \dots, C_m$, $p_1, \dots, p_m$, $L > 0$, $\tilde{\e} > 0$
        \STATE $u_l^0 := 0$, $v_l^0 := 0$, $K_l := \exp\left(-\tfrac{C_l}{L}\right)$, $l = 1, \dots, m$ 
        \REPEAT \STATE 
                $v_l^{t + 1} := \sum_{k = 1}^m w_k \ln K_k^T e^{u_k^t} - \ln K_l^T e^{u_l^t}, \quad
                \u^{t + 1} := \u^t$ 
                \STATE
                $t := t + 1$
                \STATE  $u^{t + 1}_l := \ln p_l - \ln K_l e^{v_l^t}, \quad 
                \v^{t + 1} := \v^t$ 
                \STATE $t := t + 1$
        \UNTIL{$\sum_{l = 1}^m w_l \norm{B_l^T(u_l^t, v_l^t) \one - \bar{q}^t}_1 \le \frac{\tilde{\e}}{4\max_l \norm{C_l}_\infty}$, where $B_l(u_l, v_l)=\diag\left(e^{u_l}\right) K_l \diag\left(e^{v_l}\right)$, $\bar{q}^t := \sum_{l = 1}^m w_l B_l^T(u_l^t, v_l^t) \one$}
         \STATE {$q := \tfrac{1}{\sum_{l = 1}^m w_l \la \one, B_l \one \ra} \sum_{l = 1}^m w_l B_l^T \one$}
         \STATE {Calculate $\hat{\pi}_1, \dots, \hat{\pi}_m$ by Algorithm~2 from~\cite{altschuler2017near-linear} s.t.\\
        $\hat{\pi}_l \in \U(p_l, q)$, $
            \norm{\hat{\pi}_l - B_l}_1 \le \norm{B_l \one - p_l}_1 + \norm{B_l^T \one - q}_1$.}
        \ENSURE {$q$, $\hat{\Bpi} = [\hat{\pi}_1, \dots, \hat{\pi}_m]$.}
    \end{algorithmic}
\end{algorithm}

\begin{theorem}\label{Th:ProxIBP}
Let $\bar \Bpi^N = \frac{1}{N}\sum_{k=1}^N\Bpi^k$, where $\Bpi^k$ are the iterates of \eqref{algo:prox_IBP}.
Then, after $N = \frac{4Lm\ln n}{\e}$ iterations, it holds that 

$$ \sum_{l=1}^m w_l \la C_l, \bar \pi_l^N \ra   \leq \min_{\Bpi \in \C_1 \cap \C_2} \sum_{l=1}^m w_l \la C_l,  \pi_l \ra   + \e.$$
Moreover, the accuracy $\tilde{\e}$ for the solution of \eqref{algo:prox_IBP}  is sufficient to be set as $\tilde{\e} = \widetilde O(\e^2 / (m n^3))$ and the complexity of IBP on $k$-th iteration is bounded as
\begin{gather}\label{IterationCost_IBP}
    m n^2 \widetilde{O}\left(\min\Bigg\{\exp{\left(\frac{\bar{c}_k}{L}\right) \ln\frac{\bar{c}_k}{\tilde{\e}}}, \frac{\bar{c}_k^2}{L \tilde{\e}}\Bigg\} \right), \\
    \ak{\bar{c}_k = O\left(\max_{l=1,...,m} \left[\norm{C_l}_\infty + L \ln\left(\frac{\max_{i, j} [\pi^k_l]_{i j}}{\min_{i, j} [\pi^k_l]_{i j}}\right)\right]\right)}.
\end{gather}
\end{theorem}
The proof of Theorem~\ref{Th:ProxIBP} is based on Theorem~\ref{mainTheoremDL_G} and \cite{kroshnin2019complexity}. All the remarks from Section~\ref{OT_PS} for Proximal Sinkhorn algorithm also hold for Proximal IBP. 
In \cite{kroshnin2019complexity} it was shown that complexity of IBP is $\widetilde{O} \left(n^2 / \varepsilon^2\right)$. Despite the theoretical complexity of Proximal IBP is worse than this bound, we show in the next section that in practice Proximal IBP beats the standard IBP algorithm. As an alternative to the IBP algorithm we mention primal-dual accelerated gradient descent \cite{dvurechensky2018decentralize,uribe2018distributed}.
\subsection{Numerical Illustration}

In this section, we present preliminary computational results for the numerical performance analysis of the Proximal Iterative Bregman Projection (ProxIBP) method discussed above asthe iterates~\eqref{algo:prox_IBP}.

Initially, we show the results for the computation of a non-regularized Wasserstein barycenter of a set of $10$ truncated Gaussian distributions with finite support. For the finite support $x=[-5,-4.9,-4.8,\hdots,-0.1,0,0.1,\dots,4.8,4.9,5]$, we set the finite distribution $p_l$ such that $p_l(i) = \mathcal{N}(x_i;\mu_i,\sigma_i)$, that is, the value at coordinate $i$ of the distribution $p_l$, for $1\leq l \leq m$, is the value of the Normal distribution with mean $\mu_i$ and standard deviation $\sigma_i$. The values $\{\mu_i\} \sim \text{Uniform}[-5,5]$, are uniformly chosen in the line segment $[-5,5]$, and the values are selected as $\{\mu_i\} \sim \text{Uniform}[0.25,1.25]$. For simplicity of exposition, we select uniform weighting for all distributions, i.e., $w_l = 1/m$.


Figure~\ref{fig:results_Guassians} shows the numerical results for a number of comparative scenarios between the Iterative Bregman Projection (IBP) algorithm proposed in \cite{benamou2015iterative} and its Proximal variant in~\eqref{algo:prox_IBP}. For both algorithms, we show the function values achieved by the generated iterates, and the final approximated barycenter. The results for the IBP algorithm are shown in Figure~\ref{fig:results_Guassians}(a) and Figure~\ref{fig:results_Guassians}(b). Figure~\ref{fig:results_Guassians}(a) shows the weighted distance between the generated barycenter and the original distributions for three different desired accuracy values. 
It is clear that a bigger $\e$ generates a faster convergence, but the final cost is slightly higher than in other cases. Figure~\ref{fig:results_Guassians}(b) shows the resulting barycenter for the three values of the accuracy parameter. For higher accuracy, the effects of the regularization constant are smaller and thus we obtain a ``spikier'' barycenter. Figure~\ref{fig:results_Guassians}(c) and Figure~\ref{fig:results_Guassians}(d) shows a similar analysis for the proposed Proximal IBP in~\eqref{algo:prox_IBP}, in Figure~\ref{fig:results_Guassians}(c) we observe the function value of the generated barycenter, for a fixed number of inner loop iterations, and changing values of $L$, note that here $L$ is not a regularization parameter but the weight on the Bregman function. For larger values of $L$, the inner loop problem is easier to solve, requires less iterations to achieve certain accuracy, with the price in a larger number of iterations in the outer loop. For the particular problem studied, $200$ iterations in the outer loop are sufficient to achieve good performance even with relavively smaller values of $L$. Figure~\ref{fig:results_Guassians}(c) shows the generated barycenters for the Proximal IBP algorithm. Finally, Figure~\ref{fig:results_Guassians}(e) and Figure~\ref{fig:results_Guassians}(f) show the results, for the analogous adaptive stopping condition described in Line $11$ of Algorithm~\ref{Alg:Sinkhorn} with $\e = 1\cdot 10^{-10}$. We test two different values of the parameter $L$, namely $1$ and $0.1$. Additionally, we explore the suggested adaptive search procedure, where one decreases the value of the parameter $L$ at each iteration, until the inner problem has become particularly hard to solve. This last approach is shown a fast convergence as it reaches a comparable value in around $10$ iterations. Figure~\ref{fig:results_Guassians}(f) shows the resulting barycenters.

\begin{figure}
	\centering
	\subfigure[{}]{\includegraphics[origin=c,width=0.43\textwidth]{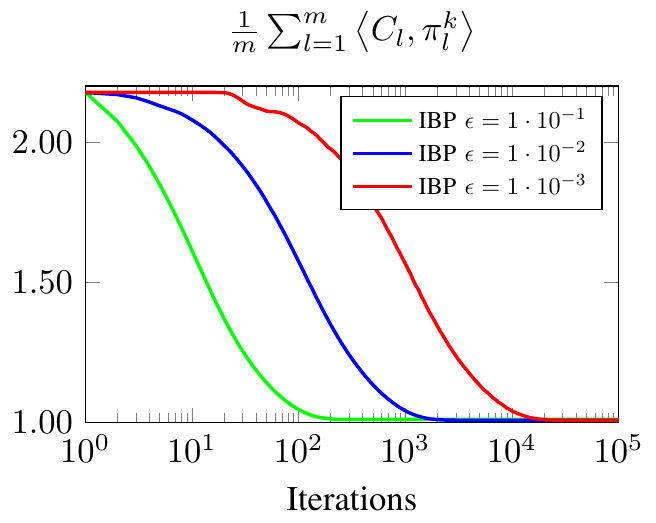}}
	\subfigure[{}]{\includegraphics[origin=c,width=0.42\textwidth]{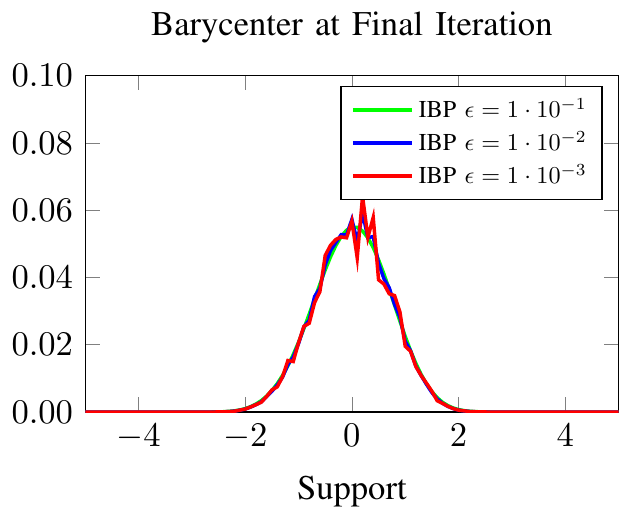}}
	\subfigure[{}]{\includegraphics[origin=c,width=0.43\textwidth]{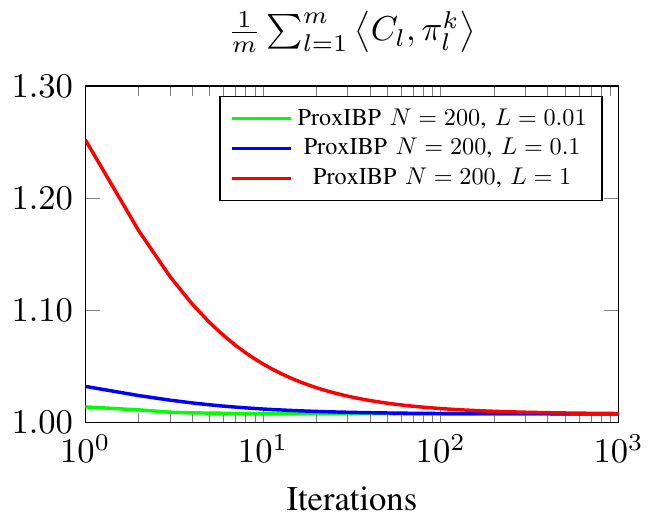}}
	\subfigure[{}]{\includegraphics[origin=c,width=0.42\textwidth]{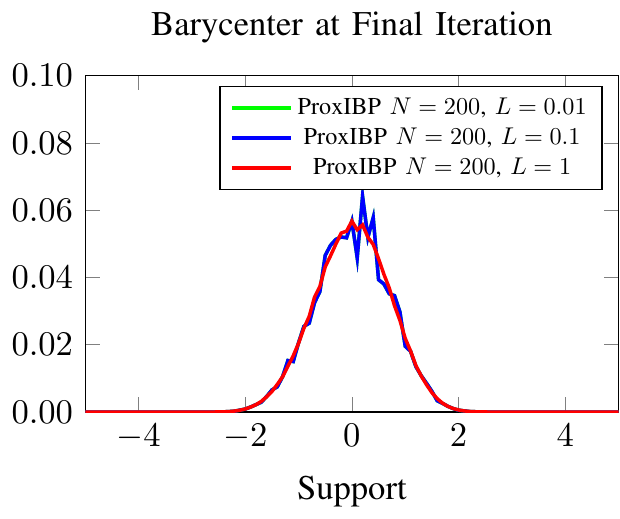}}
	\subfigure[{}]{\includegraphics[origin=c,width=0.43\textwidth]{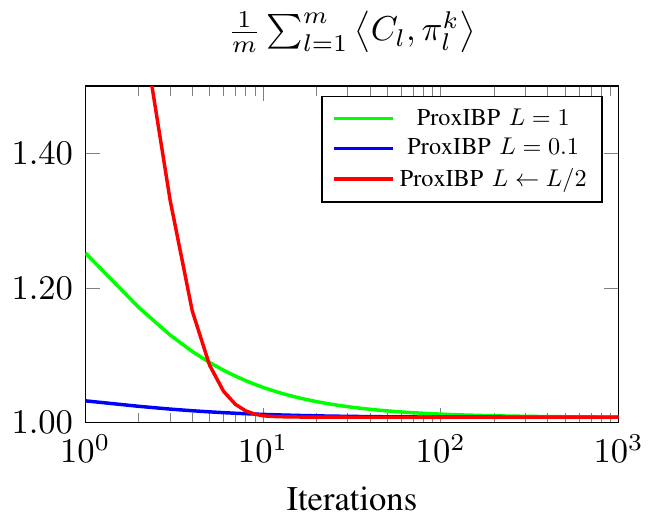}}
	\subfigure[{}]{\includegraphics[origin=c,width=0.42\textwidth]{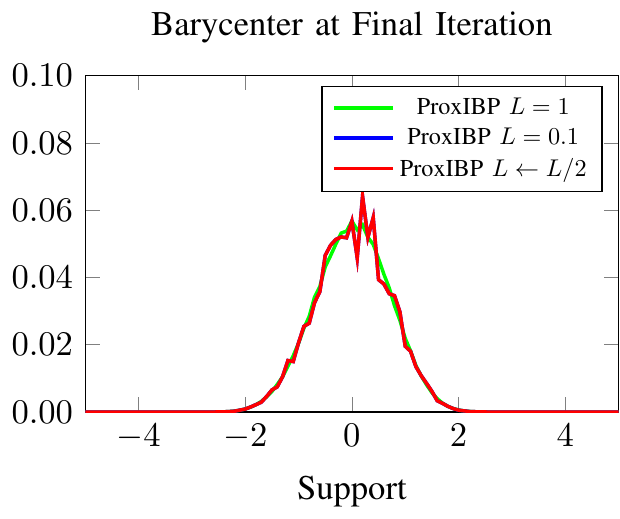}}
	\caption{Numerical results for the computation of the barycenter of $10$ truncated Gaussian random variables with finite support for the IBP Algorithm and the Proximal IBP algorithm. Both function value and final resulting barycenter are shown for an number of simulation scenarios.}
	\label{fig:results_Guassians}
\end{figure}

Figure~\ref{fig:results_seven} shows the result of applying the proximal IBP algorithm to the computation of the barycenter of $20$ images of the number $7$ from the  MNIST dataset \cite{lecun1998mnist}. As shown in Figure~\ref{fig:results_seven}(b), the Euclidean mean among the images does not preserve the geometric properties of the images, making the optimal transport distances suitable for this particular problem. Figure~\ref{fig:results_seven}(a) shows, from left to right, and top to bottom, the sequence of generated barycenters for some of the initial iterations of the algorithm. It is evident, that the generated barycenter maintains the geometric structure of the images, as the barycenter is itself a prototypical image of the number $7$. Figure~\ref{fig:results_seven}(c) shows the average distance of the generated barycenter for the first $400$ iterations of the algorithm.

\begin{figure}
	\centering
	\vspace{-3em}
	\subfigure[Iterates]{\includegraphics[origin=c,width=0.35\textwidth]{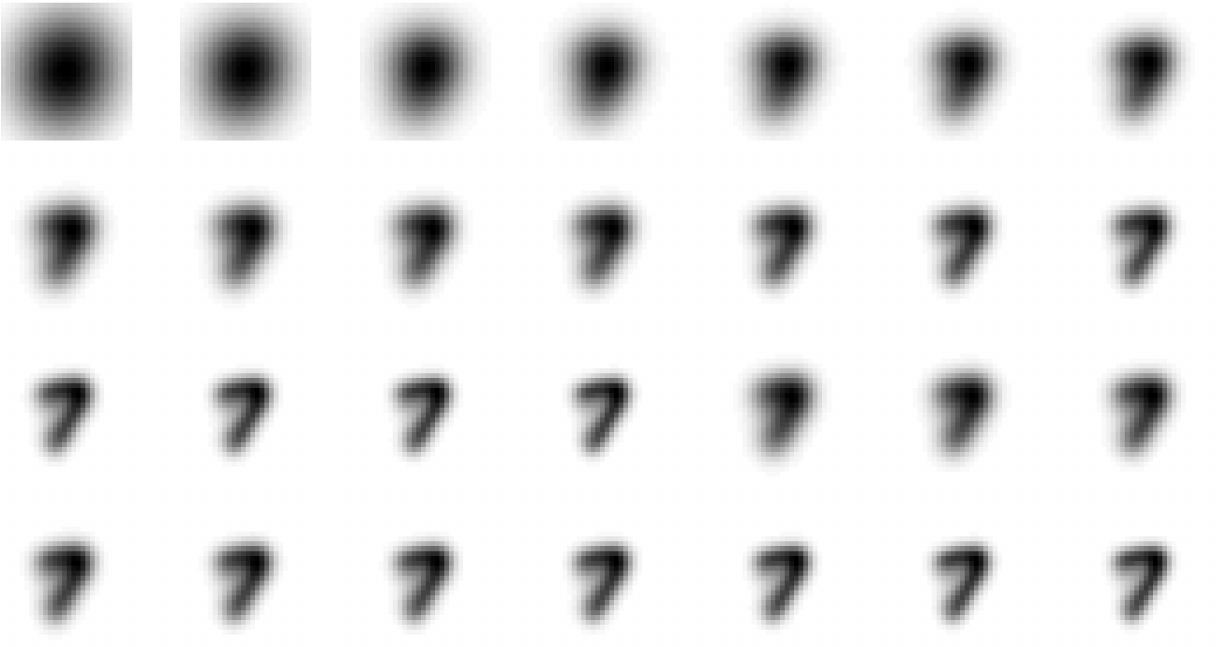}}
	\subfigure[Euclidean Mean]{\includegraphics[origin=c,width=0.30\textwidth]{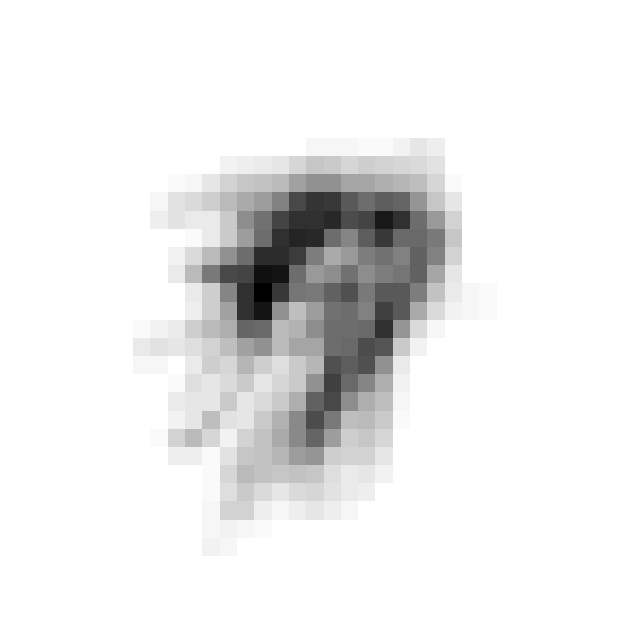}}
	\subfigure[Function Value]{\includegraphics[origin=c,width=0.3\textwidth]{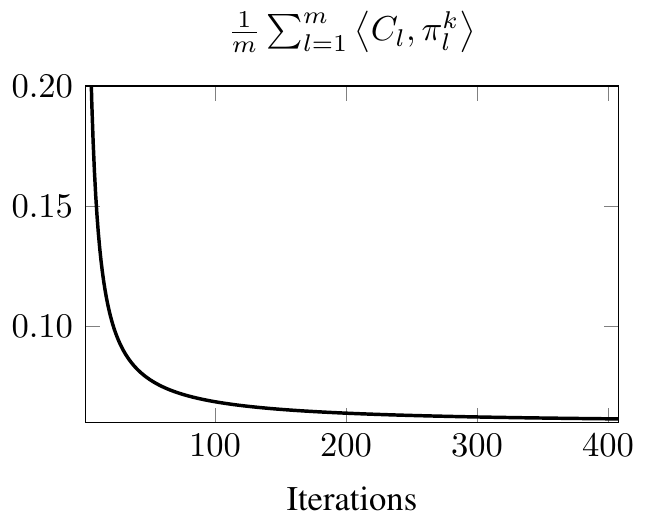}}
	\caption{Computation of the barycenter of a set of $20$ images of the number $7$ from the MNIST dataset \cite{lecun1998mnist}. }
	\label{fig:results_seven}
\end{figure}

Figure~\ref{fig:results_weights} shows the influence of the weights in the computation of the Wasserstain barycenter of four images via the Prox IBP algorithm. We have used four images corresponding to the digits $0$, $1$, $2$, and $3$, and positioned each one of them at the corner of a square. Each of the images shown inside the square corresponds to the obtained barycenter whith weights proportional to the distance to each of the corners. For example, the upper edge of the square assumes zero weights to the digits $2$ and $3$ and only shows the changes in the weights between $0$ and $1$. The center image corresponds to equal weights to all images. The other images are generated similarly.

\begin{figure}
	\centering
	\includegraphics[origin=c,width=0.4\textwidth]{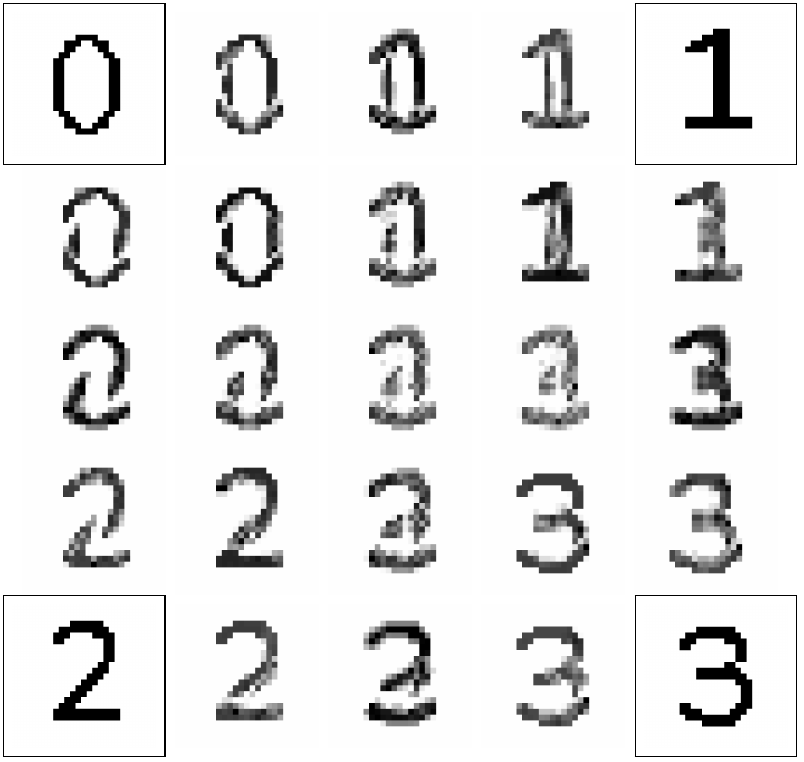}
	\caption{Approximate barycenter computed by the Prox IBP algorithm, for different weighting combinations between four original images (marked by the black boxes).  }
	\label{fig:results_weights}
\end{figure}

\textbf{Acknowledgments.} 
The work in sections \ref{OT_PS} and \ref{sec:Prox_IBP} was funded by Russian Science Foundation (project 18-71-10108). The work in section \ref{S:clustering} was supported by RFBR 18-31-20005 mol$\_$a$\_$ved.
The work of D.~Dvinskikh and D.~Pasechnyk partially conducted in Sochi-Sirius, July, 2018.

\bibliography{PD_references}

\appendix

\section{Adaptive gradient method and proof of Theorem \ref{mainTheoremDL_G}}\label{ProofGM}

\begin{algorithm}
\caption{Adaptive gradient method with inexact model of the objective}
\label{Alg1}
\begin{algorithmic}[1]
\STATE \textbf{Input:} $x_0$ is the starting point, $L_{0} > 0$ and 
$\delta>0$.
\STATE Set
 $S_0 := 0$
\FOR{$k \geq 0$}
\STATE Find the smallest $i_k\geq 0$ such that
\begin{equation}\label{exitLDL_G}
f(x_{k+1}) \leq f(x_{k}) + \psi_{\delta}(x_{k+1}, x_{k}) +L_{k+1}V[x_{k}](x_{k+1}) + \delta,
\end{equation}
where $L_{k+1} = 2^{i_k-1}L_k$, $S_{k+1} := S_k + \frac{1}{L_{k+1}}$.
\begin{equation}\label{equmir2DL_G}
\phi_{k+1}(x) := \psi_{\delta}(x, x_k)+L_{k+1}V[x_k](x), \quad
x_{k+1} := {\arg\min_{x \in Q}}^{\widetilde{\delta}} \phi_{k+1}(x).
\end{equation}
\ENDFOR
\ENSURE 
$\bar{x}_N= \frac{1}{S_N}\sum_{k=0}^{N-1}\frac{x_{k+1}}{L_{k+1}}$
\end{algorithmic}
\end{algorithm}

\def\at2#1{{\color{black}#1}}
\newcommand{\leqarg}[1]{\ensuremath{\stackrel{\text{#1}}{\leq}}}
\newcommand{\eqarg}[1]{\ensuremath{\stackrel{\text{#1}}{=}}}

Before we prove the theorem we should note that that we can reduce Algorithm~\ref{Alg1} to Algorithm~\ref{Alg1_nonadaptive}. Indeed, let us always choose $L_{k+1} = L$ instead of $L_{k+1} = 2^{i_k-1}L_k$ in Algorithm~\ref{Alg1}. In this case it is guaranteed that we exit the inner loop in  Algorithm~\ref{Alg1} after the first step due to $(\delta, L)$-model definition. Moreover, with this choice of $L_{k+1}$ Algorithm~\ref{Alg1} generates the same sequences as in Algorithm~\ref{Alg1_nonadaptive}. We prove Thus, Theorem \ref{mainTheoremDL_G} is a corollary of Theorem \ref{mainTheoremDL_G_2}.


\begin{theorem}
	\label{mainTheoremDL_G_2}
		Let $V[x_0](x_*) \leq R^2$, where $x_0$~is the starting point, and $x_*$~is the nearest minimum point to the point $x_0$ in the sense of Bregman divergence $V[y](x)$. Then, for the sequence, generated by Algorithm~\ref{Alg1} the following inequality holds:
	\begin{equation}
	f(\bar{x}_N) - f(x_*) \leq \frac{R^2}{S_N}  + \widetilde{\delta} + \delta   \leq \frac{2LR^2}{N} + \widetilde{\delta} + \delta.
	\end{equation}
Moreover, for Algorithm~\ref{Alg1} the total number of attempts to solve \eqref{equmir2DL_G} is bounded by $2N + \log_2\frac{L}{L_0}$. 
\end{theorem}

First, we need to prove two lemmas in order to obtain the final result. Let us prove Lemma \ref{MainLemma}.

	
\begin{lemma}\label{MainLemma}
	Let $\psi(x)$ be a convex function and
	\begin{gather*}
	y = {\arg\min_{x \in Q}}^{\widetilde{\delta}} \{\psi(x) + \beta V[z](x)\},
	\end{gather*}
	where $\beta \geq 0$.
Then
	\begin{equation*}
	\psi(x_*) + \beta V[z](x_*) \geq \psi(y) + \beta V[z](y) + \beta V[y](x_*) - \widetilde{\delta}.
	\end{equation*}
	\label{lemma_maxmin_2}
\end{lemma}

\begin{proof}
	Using Definition \ref{solNemirovskiy}, we have:
	\begin{gather*}
		\exists g \in \partial\psi(y), \,\,\, \langle g + \at2{\beta}\nabla_y V[z](y), x_* - y \rangle \geq -\widetilde{\delta}.
	\end{gather*}
	Then inequality
	\begin{gather*}
		\psi(x_*) - \psi(y) \geq \langle g, x_* - y \rangle \geq \langle \at2{\beta}\nabla_y V[z](y), y - x_* \rangle - \widetilde{\delta}
	\end{gather*}
and equality
	\begin{gather*}
	\langle \nabla_y V[z](y), y - x_* \rangle=\langle \nabla d(y) - \nabla d(z), y - x_* \rangle=d(y) - d(z) - \langle \nabla d(z), y - z \rangle +\\ + d(x_*) - d(y) - \langle \nabla d(y), x_* - y \rangle - d(x_*) + d(z) + \langle \nabla d(z), x_* - z \rangle=\\=
	V[z](y) + V[y](x_*) - V[z](x_*)
	\end{gather*}
complete the proof.
\end{proof}

\begin{lemma}
\label{lemma_maxmin_3DL_G}
    We have the following inequality:
	\begin{gather*}
		\frac{f(x_{k+1})}{L_{k+1}} - \frac{f(x_*)}{L_{k+1}}\leq V[x_k](x_*) - V[x_{k+1}](x_*) + \at2{\frac{\widetilde{\delta}}{L_{k+1}}} + \frac{\delta}{L_{k+1}}.
	\end{gather*}
\end{lemma}

\begin{proof}
    From the stopping criterion \eqref{exitLDL_G}, we have:
	\begin{gather*}
	f(x_{k+1}) \leq f(x_{k}) + \psi_{\delta}(x_{k+1}, x_{k}) + L_{k+1} V[x_{k}](x_{k+1}) + \delta.
	\end{gather*}
	Using Lemma \ref{lemma_maxmin_2} with $\psi(x)=\psi_{\delta}(x, x_{k})$ \at2{and $\beta=L_{k+1}$,} we obtain:
    \begin{gather*}
	 f(x_{k+1}) \leq f(x_{k}) + \psi_{\delta}(x_*,x_{k})
	 	 + L_{k+1}V[x_k](x_*) - L_{k+1}V[x_{k+1}](x_*) + \at2{\widetilde{\delta}} + \delta.
	\end{gather*}
    In view of the model definition \eqref{model_def}, we have:
	\begin{gather*}
	f(x_{k+1}) \leq f(x_*) + L_{k+1}V[x_k](x_*) - L_{k+1}V[x_{k+1}](x_*) + \at2{\widetilde{\delta}} + \delta.
	\end{gather*}
\end{proof}

\begin{remark}
\leavevmode
\label{remark_maxmin}
Let us show that
$L_{k} \leq 2L\quad \forall k \geq 0$.
For $k=0$ this is true from the fact that $L_0 \leq L$. For $k \geq 1$ this follows from the fact that we leave the inner cycle earlier than $L_{k}$ will be greater than $2L$. The exit from the cycle is guaranteed by the condition that there is an $(\delta, L)$-model for $f(x)$ at any point $x \in Q$.
\end{remark}

Finally, we prove the theorem.

\begin{proof}

Let us sum up the inequality from Lemma \ref{lemma_maxmin_3DL_G} from $0$ to $N - 1$:
	\begin{gather*}
		\sum_{k=0}^{N-1}\frac{f(x_{k+1})}{L_{k+1}} - S_N f(x_*) \leq V[x_0](x_*) - V[x_N](x_*) + S_N\widetilde{\delta} + S_N\delta.
	\end{gather*}
	Since $V[x_N](x_*) \geq 0$ and $V[x_0](x_*) \leq R^2$, we obtain inequality
	\begin{gather*}
		\sum_{k=0}^{N-1}\frac{f(x_{k+1})}{L_{k+1}} - S_N f(x_*) \leq R^2 + S_N\widetilde{\delta} + S_N\delta.
	\end{gather*}
	Let us divide both parts by $S_N$.
	\begin{gather*}
	\frac{1}{S_N}\sum_{k=0}^{N-1}\frac{f(x_{k+1})}{L_{k+1}} - f(x_*) \leq \frac{R^2}{S_N} + \widetilde{\delta} + \delta.
	\end{gather*}
	Using the convexity of $f(x)$ we can show that
	\begin{gather*}
		f(\bar{x}_N) - f(x_*) \leq \frac{R^2}{S_N} + \widetilde{\delta} + \delta.
	\end{gather*}
	Remains only to prove that $$\frac{1}{S_N} \leq \frac{2L}{N}.$$ 
	As it follows from Definition~\ref{model} and Remark~\ref{remark_maxmin} for all $k \geq 0$
$L_{k} \leq 2L$.
	Thus,  we have that $$\frac{1}{L_k} \geq \frac{1}{2L}$$ and $$S_N = \sum_{k=0}^{N}\frac{1}{L_k} \geq \frac{N}{2L}.$$

The total number of attempts to solve \eqref{equmir2DL_G} is bounded in the same way as in \cite{nesterov2015universal}.
\end{proof}

In the same way as it is done in \cite{anikin2017dual,chernov2016fast,dvurechensky2016primal-dual,dvurechensky2017adaptive,nesterov2018primal-dual}, one can show that the proposed method is primal-dual.

\section{Analysis of Algorithm \ref{Alg2} in the case of $(\delta,L,\mu)$-model}\label{sec: ProofThFedyor}

Now we consider a proof of Theorem \ref{thm_fedyor}. \pd{To analyze Algorithm \ref{Alg2}, assume that it works for $k$ iterations.} By Lemma \ref{MainLemma}\pd{,} for each $x \in Q$:
$$ -\widetilde{\delta} \leq \psi_\delta(x,x^k) - \psi_\delta(x^{k+1},x^k) + L^{k+1}V[x^k](x)-L^{k+1}V[x^{k+1}](x)-L^{k+1}V[x^k](x^{k+1}).$$
It means that
	\begin{equation} \label{equat1}
		L_{k+1}V[x^{k+1}](x)\leq\widetilde{\delta} + \psi_\delta(x,x^k)-\psi_\delta(x^{k+1}, x^k) +
	\end{equation}
	$$
	+ L_{k+1}V[x^k](x)-L_{k+1}V[x^k](x^{k+1}).
	$$
Further, $\psi_{\delta}(x, y)$ is a $(\delta, L)$-model w.r.t. $V[y](x)$ and from
	$$ f(x^{k+1}) \leq f(x^k)+\psi_\delta(x^{k+1},x^k)+L_{k+1}V[x^k](x^{k+1}) + \delta,$$
	we get
	$$-L_{k+1}V[x^k](x^{k+1}) \leq \delta - f(x^{k+1}) + f(x^k) + \psi_\delta(x^{k+1}, x^k) .$$
	 
Now \eqref{equat1} means
	\begin{equation}
		\label{eq2}
		L_{k+1}V[x^{k+1}](x) \leq \widetilde{\delta} + \delta - f(x^{k+1}) + f(x^k) + \psi_\delta(x, x^k) + L_{k+1}V[x^k](x).
	\end{equation}
	
\pd{Since} $\psi_{\delta}(x, y)$ is a $(\delta, L, \mu)$-model for $f$, we have: 
	$$f(x^k) + \psi_\delta(x,x^k) \leq f(x) - \mu V[x^k](x).$$
Considering \eqref{eq2}, we obtain:
	\begin{equation}
		\label{eq3}
		L_{k+1}V[x^{k+1}](x) \leq \widetilde{\delta} + \delta + f(x) - f(x^{k+1})+(L_{k+1}-\mu)V[x^k](x).
	\end{equation}
	
\pd{Set} $x = x_*$. \pd{Since $L_0 \le 2L$, }we have $L_{k+1} \le 2L$ for each $k \ge 0$. We also assume in Algorithm \ref{Alg2} that $L_{k+1} \ge \mu$. Thus, we have 
$$\dfrac{1}{2L} \le \dfrac{1}{L_{k+1}} \le \dfrac{1}{\mu} \quad (\forall  k = 0,1,2 \ldots).$$

Then we have $\forall i \in \mathbb{N}: i < k$
	\begin{equation}\label{eq1}
	\left(1 - \dfrac{\mu}{L_{k+1}}\right)\left(1 - \dfrac{\mu}{L_{k}}\right)\ldots\left(1 - \dfrac{\mu}{L_{k-i}}\right) \le \left(1 - \dfrac{\mu}{2L}\right)^{i+1}.
	\end{equation}
	Therefore, we obtain:
	\begin{gather*}
			V[x^{k+1}](x_*) \le \frac{f(x_*) - f(x^{k+1}) + \delta +\widetilde{\delta}}{L^{k+1}} + \left(1 - \dfrac{\mu}{L_{k+1}}\right)V[x^k](x_*),
	\end{gather*}	
	and 
	\begin{gather*}	
		\frac{f(x^{k+1}) - f(x_*)}{L_{k+1}}	+ V[x^{k+1}](x_*) \le \\
			\le \dfrac{\delta+\widetilde{\delta}}{L_{k+1}} + \left({1 - \dfrac{\mu}{L_{k+1}}}\right)V[x^k](x_*)\le (\delta+\widetilde{\delta})\left(\dfrac{1}{L_{k+1}}+\dfrac{1}{L_k}\left(1-\dfrac{\mu}{L_{k+1}}\right)\right)+ \\ +\left(1-\dfrac{\mu}{L_{k+1}}\right)\left(1-\dfrac{\mu}{L_{k}}\right)V[x^k](x_*) \le \ldots \le (\delta+\widetilde{\delta})\left(\dfrac{1}{L_{k+1}}+\dfrac{1}{L_k}\left(1-\dfrac{\mu}{L_k}\right)+ \right.\\
			+\left. \dfrac{1}{L_{k-1}}\left(1 - \dfrac{\mu}{L_k}\right)\left(1 - \dfrac{\mu}{L_{k-1}}\right)+ \ldots + \dfrac{1}{L_1}\left(1-\dfrac{\mu}{L_k}\right)\left(1-\dfrac{\mu}{L_{k-1}}\right)\ldots\left(1-\dfrac{\mu}{L_1}\right)\right) + \\ +\left(1-\dfrac{\mu}{L_{k+1}}\right)\left(1-\dfrac{\mu}{L_k}\right)\ldots \left(1-\dfrac{\mu}{L_1}\right)V[x^0](x_*).
	\end{gather*}

For further reasoning we introduce average parameter $\hat{L}$:
	$$1-\dfrac{\mu}{\hat{L}} = \sqrt[k+1]{ \left(1-\dfrac{\mu}{L^{k+1}}\right)\left(1-\dfrac{\mu}{L_{k}}\right)\ldots\left(1-\dfrac{\mu}{L_{1}}\right)}.
	$$
Note that by $L_i \ge \mu~(i=1,2,\ldots)$ 
	$$\min\limits_{1\le i \le k+1}L_i \le \hat{L} \le \max\limits_{1\le i \le k+1}L_i \pd{\leq 2L}.$$
	Now, taking into account \eqref{eq1}, we have:
	\begin{equation}\label{fedyor_str1}
\frac{f(x^{k+1}) - f(x_*)}{L_{k+1}} + V[x^{k+1}](x_*) \le \frac{\delta+\widetilde{\delta}}{\mu}\sum\limits_{i=0}^k\left(1-\dfrac{\mu}{2L}\right) + \left(1-\dfrac{\mu}{\hat{L}}\right)^{k+1}V[x^0](x_*) \le 
\end{equation}
\begin{equation}\label{fedyor_str2}
\le \dfrac{2L(\delta+\widetilde{\delta})}{\mu^2}\left(1 - \left(1-\dfrac{\mu}{2L}\right)^{k+1}\right) + \left(1-\dfrac{\mu}{\hat{L}}\right)^{k+1}V[x^0](x_*).
	\end{equation}

Finally, we have
\begin{equation}
V[x^{k+1}](x_*) \le \dfrac{2L(\delta+\widetilde{\delta})}{\mu^2}\left(1 - \left(1-\dfrac{\mu}{2L}\right)^{k+1}\right) + \left(1-\dfrac{\mu}{\hat{L}}\right)^{k+1}V[x^0](x_*).
\end{equation}
and by \eqref{fedyor_str1} -- \eqref{fedyor_str2} and $L^{k+1} \le 2L$ means:
\begin{equation}
f(x^{k+1}) - f(x_*) \le  \dfrac{4L^2(\delta+\widetilde{\delta})}{\mu^2}\left(1 - \left(1-\dfrac{\mu}{2L}\right)^{k+1}\right) + 2L\left(1-\dfrac{\mu}{\hat{L}}\right)^{k+1}V[x^0](x_*).
\end{equation}

\section{Analysis of Algorithm \ref{Alg1_nonadaptive} in the case of $(\delta,L,\mu)$-model}\label{sec: ProofTh}

\begin{theorem}\label{agaf_Th}
Let $\psi_{\delta}(x, y)$ be a $(\delta, L, \mu)$-model for $f$ w.r.t. $V[y](x)$ and $y_k = \argmin_{i = 1,..., k}(f(x_i))$. Then, after $k$ iterations of Algorithm \ref{Alg1_nonadaptive}, we have \begin{equation}\label{agaf_eq5}
V[x^{k+1}](x_*) \leq \dfrac{\delta+\widetilde{\delta}}{\mu} + \left(1-\dfrac{\mu}{L}\right)^{k+1}V[x^0](x_*).
\end{equation}
and
\begin{equation}\label{agaf_eq4}
f(y_{k+1}) - f(x_*) \le L\left(1-\frac{\mu}{L}\right)^{k+1}V[x^0](x_*) + \delta + \widetilde{\delta}.
\end{equation}
\end{theorem}
\begin{proof}
Clearly, $f(x_*) \leq f(x^{k+1})$ and
$$LV[x^{k+1}](x_*) \leq \widetilde{\delta} + \delta + (L-\mu)V[x^k](x_*),$$
i.e.
$$V[x^{k+1}](x_*) \leq \dfrac{1}{L}(\delta + \widetilde{\delta}) + \left(1 - \dfrac{\mu}{L}\right)V[x^k](x_*).$$
Further,
\begin{align*}
	V[x^{k+1}](x_*) \leq \dfrac{1}{L}(\delta+\widetilde{\delta}) + \left( 1-\dfrac{\mu}{L}\right)\left(\dfrac{1}{L}(\delta + \widetilde{\delta}) + \left(1 - \dfrac{\mu}{L}\right)V[x^{k-1}](x_*)\right) \leq \ldots \leq \\
	\leq \dfrac{1}{L}(\widetilde{\delta} + \delta)\left(1+ \left(1-\dfrac{\mu}{L}\right)+ \ldots + \left(1-\dfrac{\mu}{L}\right)^k \right) + \left(1-\dfrac{\mu}{L}\right)^{k+1}V[x^0](x_*).
\end{align*}

Therefore, taking into account the following fact
$$\sum\limits_{i=0}^k\left(1-\frac{\mu}{L}\right)^i < \frac{1}{1 - \left(1 - \frac{\mu}{L}\right)} = \frac{L}{\mu},$$
we obtain
\begin{equation}\label{agaf_eq5}
V[x^{k+1}](x_*) \leq \dfrac{\delta+\widetilde{\delta}}{\mu} + \left(1-\dfrac{\mu}{L}\right)^{k+1}V[x^0](x_*).
\end{equation}

Now we consider the question on convergence by function: 
\begin{gather*}
	V[x^{k+1}](x_*) \leq \left(f(x_*) - f(x^{k+1})+\delta+\widetilde{\delta}\right)\dfrac{1}{L}+\left(1-\dfrac{\mu}{L}\right)V[x^k](x_*) \leq\\
	\leq \left(f(x_*) - f(x^{k+1})+\delta+\widetilde{\delta}\right)\dfrac{1}{L} + \\ + \left(1-\dfrac{\mu}{L}\right)\left(\left(f(x_*) - f(x^{k})+\delta+\widetilde{\delta}\right)\dfrac{1}{L} + \left(1-\dfrac{\mu}{L}\right)V[x^{k-1}](x_*)\right) \leq\\
	\le \ldots \le \left(1-\dfrac{\mu}{L}\right)^{k+1} V[x^0](x_*)+ \dfrac{1}{L}\sum\limits_{i=0}^k \left(1-\dfrac{\mu}{L}\right)^i\left(f(x_*) - f\left(x^{k+1-i}\right)+\delta+\widetilde{\delta}\right).
\end{gather*}

Therefore, we have
$$\dfrac{1}{L}\sum\limits_{i=0}^k\left(1-\dfrac{\mu}{L}\right)^i(f(x^{k+1-i})-f(x_*)) \leq \left(1-\dfrac{\mu}{L}\right)^{k+1}V[x^0](x_*)+\dfrac{1}{L}\sum\limits_{i=0}^k\left(1-\dfrac{\mu}{L}\right)^i(\delta+\widetilde{\delta}).$$

Denote by $y_k = \argmin_{i = 1,..., k}(f(x_i))$. Then, taking into account
$$\frac{1}{L}\sum\limits_{i=0}^k\left(1-\frac{\mu}{L}\right)^i =  \frac{1}{\mu}\left(1-\left(1 - \frac{\mu}{L}\right)^{k+1}\right) \geq \frac{1}{L},$$
we obtain
\begin{equation}\label{agaf_eq4}
f(y_{k+1}) - f(x_*) \le \mu \dfrac{\left(1 - \frac{\mu}{L}\right)^{k+1}}{1-\left(1-\frac{\mu}{L}\right)^{k+1}}V[x^0](x_*) + \delta + \widetilde{\delta} \le 
\end{equation}
$$
\le L\left(1-\frac{\mu}{L}\right)^{k+1}V[x^0](x_*) + \delta + \widetilde{\delta}.
$$
\end{proof}

\section{Some Numerical Tests for Algorithms \ref{Alg1_nonadaptive} and \ref{Alg2}}

\pd{We consider} two numerical examples for Algorithms \ref{Alg1_nonadaptive} and \ref{Alg2} for minimizing $\mu$-strongly convex \pd{objective function} of $N$ variables on a  unit ball $B_1(0)$ with center at zero with respect to the standard Euclidean norm.  It is clear that such \pd{functions} admit $(\delta, L, \mu)$-model of the standard form $\psi_{\delta}(x, y) = \langle \nabla f(y), x - y \rangle$ for the case of Lipschitz-continuous gradient $ \nabla f$. In the first of the considered examples, it is easy to estimate $L$ and $\mu$, and the ratio $\frac{\mu}{L}$ is not very small, which ensures a completely acceptable rate of convergence of the non-adaptive method (see Table \ref{tab1} below).  In the second example, \pd{the objective is ill-conditioned meaning that} the ratio $\frac{\mu}{L}$ so small that the computer considers the value $1 - \frac{\mu}{L}$ to be equal to 1 and Theorem \ref{agaf_Th} for the non-adaptive algorithm does not allow to estimate the rate of convergence at all. In this case, the use of adaptive Algorithm \ref{Alg2} leads to noticeable results (see the Table \ref{tab2} below).

\begin{example}\label{fedyor_ex_1}
Consider a function
$$f(x)=x_1^2+2x_2^2+3x_3^2+\ldots+Nx_N^2,$$
where $N=100$ and input data $$x^0=\frac{(0.2,\ldots,0.2)}{||(0.2,\ldots,0.2)||} \text{ is the initial approximation },$$
$\mu=2$, $L^0=2\mu$, $L=2N$.

\begin{table}
\centering
\caption{Results for Example \ref{fedyor_ex_1}.}
\begin{tabular}{|c|c|c|c|c|}
\hline
\textbf{}    & \multicolumn{2}{c|}{\textbf{Non-adaptive}} & \multicolumn{2}{c|}{\textbf{Adaptive}} \\ \hline
\textbf{k} & \textbf{Time}      & \textbf{Estimate}     & \textbf{Time}    & \textbf{Estimate}   \\ \hline
160          & 0:01:19            & 0.19827               & 0:05:25          & 0.02110             \\ \hline
180          & 0:01:27            & 0.16220               & 0:05:55          & 0.01258             \\ \hline
200          & 0:01:36            & 0.13264               & 0:07:11          & 0.00750             \\ \hline
220          & 0:01:55            & 0.10849               & 0:07:19          & 0.00474             \\ \hline
240          & 0:01:57            & 0.08873               & 0:07:56          & 0.00282             \\ \hline
\end{tabular}
\label{tab1}
\end{table}

The results of the comparison of the work of algorithms 1 and 3 are presented in the comparative Table \ref{tab1}, where $k$~is the number of iterations of these algorithms.
\end{example}

As we can see from the Table 1, in the previous example the non-adaptive method converges no worse than the adaptive one. However, it is possible that $\frac{\mu}{L}$ is too small, which leads to $1-\frac{\mu}{L}\approx1$. In this case, Theorem \ref{agaf_Th} cannot estimate the rate of convergence of the method. We give another example.

\begin{example}\label{fedyor_ex_2}
Consider the target functional
$$f(x_1,\ldots,x_N)=\sum_{k=1}^N(kx_k^2+e^{-kx_k}).$$

It is easy to verify that for such a \pd{function} one can choose $\mu=2+\frac{1}{e}$ and $L=2N+N^2e$ and the program calculates the value of $1-\frac{\mu}{L}$ equal to 1. However, applying Algorithm \ref{Alg2} with adaptive tuning to the constant $L$ and Theorem \ref{thm_fedyor} we obtain meaningful results, which we present in Table \ref{tab2}.

\begin{table}
\centering
\caption{Results for Example \ref{fedyor_ex_2}.}
\begin{tabular}{|c|c|c|}
\hline
\textbf{}  & \multicolumn{2}{c|}{\textbf{Adaptive}} \\ \hline
\textbf{k} & \textbf{Time}    & \textbf{Estimate}   \\ \hline
50         & 0:07:37          & 0.71273             \\ \hline
100        & 0:14:27          & 0.51241             \\ \hline
150        & 0:23:00          & 0.372301            \\ \hline
200        & 0:28:07          & 0.27334             \\ \hline
250        & 0:34:32          & 0.19699             \\ \hline
300        & 0:43:10          & 0.14456             \\ \hline
\end{tabular}
\label{tab2}
\end{table}

\end{example}

Experiments were performed using CPython 3.7 software on a computer with a 3-core AMD Athlon II X3 450 processor with a clock frequency of 803.5 MHz per core. The computer's RAM was 8 GB.

\section{Complexity Analysis of Sinkhorn's Algorithm}\label{sec:Sinkhorn}

Let us consider regularized optimal transport problem
\begin{equation}\label{prob:reg_OT}
    \la C, \pi \ra + \gamma \sum_{i, j} \pi_{i j} \ln \pi_{i j} \to \min_{\pi \in \U(p, q)}.
\end{equation}
Recall that the dual problem to~\eqref{prob:reg_OT} is equivalent to 
\begin{equation}\label{prob:dual_OT}
    f(u, v) := \la \one B(u, v) \one \ra - \la u, p \ra - \la q, v \ra \to \min_{u, v \in \R^n},
\end{equation}
where $B(u, v) := \diag(e^u) e^{-C / \gamma} \diag(e^v)$ \cite{dvurechensky2018computational}. Below we present a slightly refined complexity analysis of Algorithm~\ref{Alg:Sinkhorn} based on the same approach as in \cite{dvurechensky2018computational}.
First, we prove that Sinkhorn's iterations are contractant for $e^{u^t - u^*}$ and $e^{v^t - v^*}$ in Hilbert's projective metric (cf. \cite{franklin1989scaling}).

\begin{lemma}\label{lemma:R_t}
    Let us define
    \begin{equation}
        R_t := 
        \begin{cases}
            \max_j (v_j^t - v_j^*) - \min_j (v_j^t - v_j^*), & t \bmod 2 = 0,\\
            \max_i (u_i^t - u_i^*) - \min_i (u_i^t - u_i^*), & t \bmod 2 = 1,
        \end{cases}
    \end{equation}
    where $(u^*, v^*)$ is the solution of problem~\eqref{prob:dual_OT}. Then for any $t \ge 0$ it holds $R_{t + 1} \le R_t$.
\end{lemma}

\begin{proof}
    W.l.o.g. consider even $t$. Let us denote $\pi^* = B(u^*, v^*) \in \U(p, q)$. Then for any $i$
    \[
        u^{t + 1}_i - u^*_i
        = u^t_i - u^*_i + \ln p_i - \ln\left(\sum_j e^{u^t_i - u^*_i} \pi^*_{i j} e^{v^t_j - v^*_j}\right) 
        = - \ln\left(\sum_j \frac{\pi^*_{i j}}{p_i} e^{v^t_j - v^*_j}\right),
    \]
    and since $\sum_j \frac{\pi^*_{i j}}{p_i} = 1$ one obtains
    \[
    e^{\min_j (v^t_j - v^*_j)} 
    \le \sum_j \frac{\pi^*_{i j}}{p_i} e^{v^t_j - v^*_j} 
    \le e^{\max_j (v^t_j - v^*_j)},
    \]
    therefore,
    \[
    R_{t+1} \le R_t.
    \]
\end{proof}

Now repeating the proof of Theorem~1 from \cite{dvurechensky2018computational} we obtain the following complexity bound.
\begin{theorem}\label{thm:Sinkhorn_complexity}
    The inner cycle of Algorithm~\ref{Alg:Sinkhorn} stops in number of iterations
    \[
    N = O\left(\frac{R_0}{\e'}\right),
    \]
    and 
    \[
    R_0 \le \frac{\max_{i, j} C_{i j} - \min_{i, j} C_{i j}}{\gamma}.
    \]
\end{theorem}

Nptice that now we require an approximated solution of regularized problem~\eqref{prob:reg_OT}, thus the choice of $\e'$ in Algorithm~\ref{Alg:Sinkhorn} differs from the one from \cite{altschuler2017near-linear,dvurechensky2018computational}.
\begin{theorem}\label{thm:OT_accuracy}
    Algorithm~\ref{Alg:Sinkhorn} returns $\hat\pi \in \U(p, q)$ s.t.
    \[
    \la C, \hat\pi \ra + \gamma \sum_{i, j} \hat\pi_{i j} \ln \hat\pi_{i j} 
    \le \la C, \pi^* \ra + \gamma \sum_{i, j} \pi^*_{i j} \ln \pi^*_{i j} + \tilde\e,
    \]
    where $\pi^*$ is the solution of problem~\eqref{prob:reg_OT}.
\end{theorem}

\begin{proof}
    Notice that for any $\pi \in \U(B(u,v) \one, B(u,v)^T \one)$ it holds
    \[
    \la C, B(u, v) \ra + \gamma \sum_{i, j} B(u,v)_{i j} \ln B(u,v)_{i j} 
    \le \la C, \pi \ra + \gamma \sum_{i, j} \pi_{i j} \ln \pi_{i j}. 
    \]
    It is easy to see that for any pair $\pi, \tilde\pi \in S_{n \times n}(1)$
    \[
    \abs{\sum_{i, j} \pi_{i j} \ln \pi_{i j} - \sum_{i, j} \tilde\pi_{i j} \ln \tilde\pi_{i j}} 
    \le n^2 h \ln \frac{1}{h} + \norm{\pi - \tilde\pi}_1 \ln \frac{1}{h} \quad \forall h \in (0, e^{-1}),
    \]
    thus
    \[
    \abs{\sum_{i, j} \pi_{i j} \ln \pi_{i j} - \sum_{i, j} \tilde\pi_{i j} \ln \tilde\pi_{i j}} 
    \le 2 \norm{\pi - \tilde\pi}_1 \ln\left(\frac{n^2}{\norm{\pi - \tilde\pi}_1}\right).
    \]
    Now, for any $\pi \in S_{n \times n}(1)$ and $r, c \in S_n(1)$ there exists $\tilde\pi \in \U(r, c)$ given by Algorithm~2 from \cite{altschuler2017near-linear} s.t.\ $\norm{\pi - \tilde\pi}_1 \le \norm{\pi \one - r}_1 + \norm{\pi^T \one - c}_1$.
    Combining all these facts together we obtain for $\hat\pi$ defined in Algorithm~\ref{Alg:Sinkhorn} the following estimate
    \[
    \la C, \hat\pi \ra + \gamma \sum_{i, j} \hat\pi_{i j} \ln \hat\pi_{i j}
    \le \la C, \pi^* \ra + \gamma \sum_{i, j} \pi^*_{i j} \ln \pi^*_{i j} + 2 \left(\max_{i, j} C_{i j} - \min_{i, j} C_{i j} + 2 \gamma \ln\left(\frac{n^2}{\e'}\right)\right) \e'
    \]
    Substituting 
    \[
    \e' = \frac{\tilde{\e}}{4 \left(\max_{i, j} C_{i j} - \min_{i, j} C_{i j} + 2 \gamma \ln\left(\frac{4 \gamma n^2}{\tilde{\e}}\right)\right)}
    \]
    we obtain
    \begin{align*}
        \la C, \hat\pi \ra &+ \gamma \sum_{i, j} \hat\pi_{i j} \ln \hat\pi_{i j} - \left[\la C, \pi^* \ra + \gamma \sum_{i, j} \pi^*_{i j} \ln \pi^*_{i j}\right] \\
        & \le 2 \left(\max_{i, j} C_{i j} - \min_{i, j} C_{i j} + 2 \gamma \ln\left(\frac{n^2}{\e'}\right)\right) \frac{\tilde{\e}}{4 \left(\max_{i, j} C_{i j} - \min_{i, j} C_{i j} + 2 \gamma \ln\left(\frac{4 \gamma n^2}{\tilde{\e}}\right)\right)} \\
        & \le \frac{\tilde{\e}}{2} \frac{\max_{i, j} C_{i j} - \min_{i, j} C_{i j} + 2 \gamma \ln\left(\frac{4 \gamma n^2}{\tilde{\e}}\right) + 2 \gamma \ln \frac{\tilde{\e}}{4 \gamma \e'}}{\max_{i, j} C_{i j} - \min_{i, j} C_{i j} + 2 \gamma \ln\left(\frac{4 \gamma n^2}{\tilde{\e}}\right)} \\
        & \le \frac{\tilde{\e}}{2} \left(1 + \frac{\tilde{\e} / (2 e \e')}{\max_{i, j} C_{i j} - \min_{i, j} C_{i j} + 2 \gamma \ln\left(\frac{4 \gamma n^2}{\tilde{\e}}\right)}\right) \le \tilde{\e}.
    \end{align*}
\end{proof}

\section{Additional Experiments for Prox-Sinkhorn Algorithm}
\label{S:app_prox_Sinkh_experim}

Figure \ref{fig:sinkhorn_inner} shows the dependence of the mean inner method iteration number upon accuracy and size of the vector $\pi$. With the growth of $L$, there is a decrease in the mean inner method iteration number. However, the type of dependence from the accuracy or size of the problem is the same. 

\begin{figure}[H]
    \centering
    \includegraphics[width=0.86\textwidth]{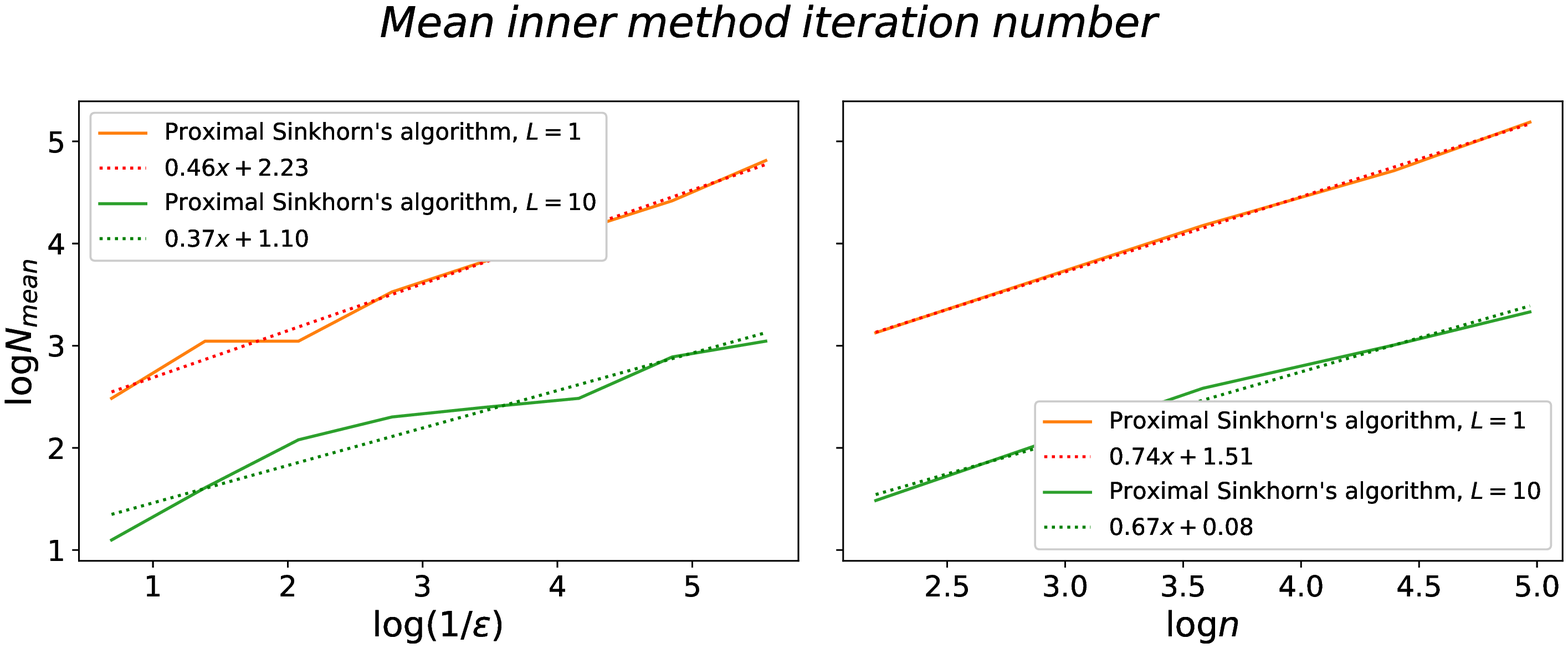}
    \caption{Comparison of inner method iteration number of proximal Sinkhorn's algorithm for different $L$. }
    \label{fig:sinkhorn_inner}
\end{figure}

Consider a graph of change of the auxiliary problem solution complexity with increasing external method iteration number (fig. \ref{fig:sinkhorn_inner2}). Note that at the first interval there is an increase in the inner method iteration number with two peaks at different levels. On subsequent iterations of the external method the complexity of the solution of the auxiliary problem decreases, approaching a constant.

\begin{figure}[H]
    \centering
    \includegraphics[width=0.48\textwidth]{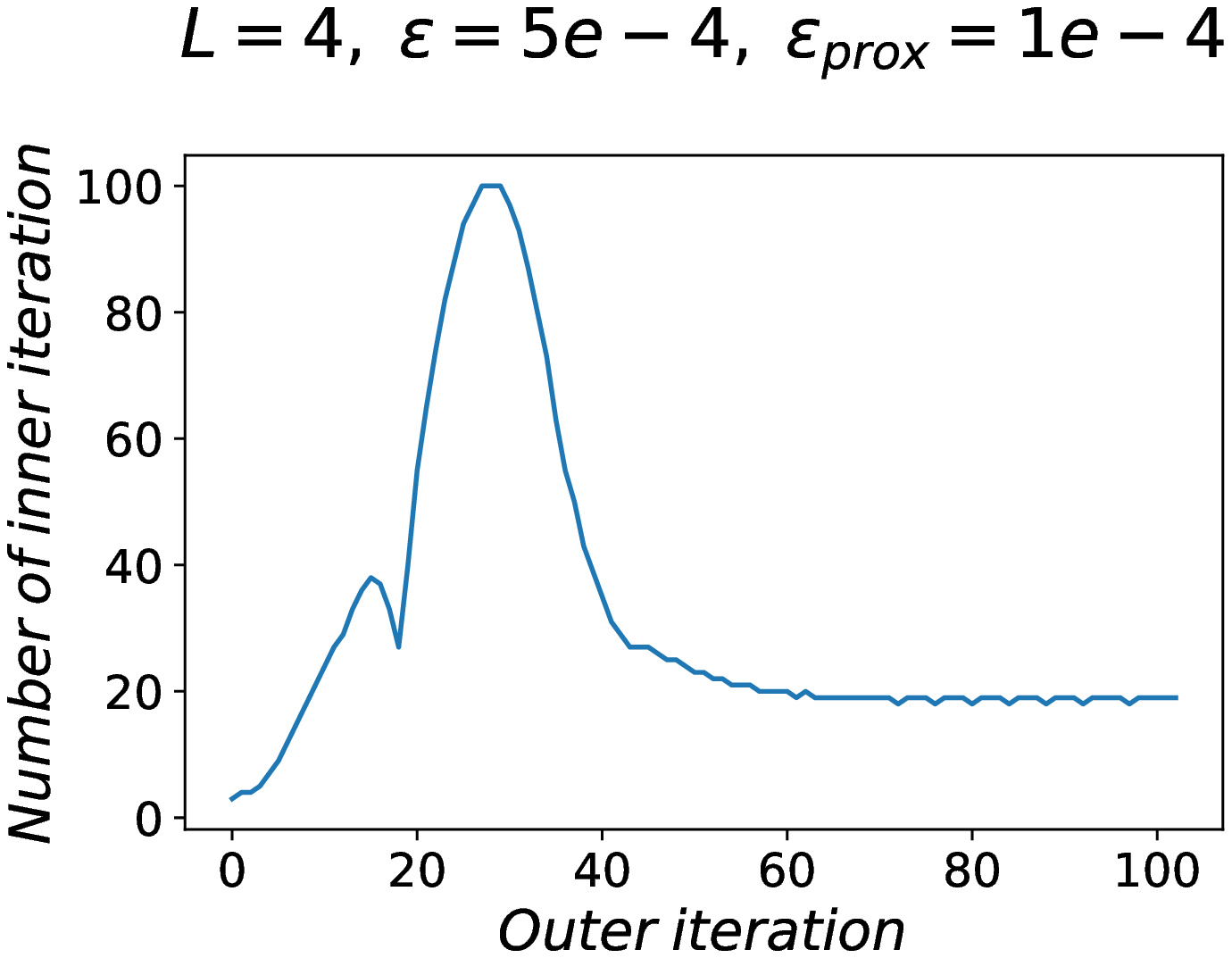}
    \caption{The dependence of the total \textit{Number of inner iteration} from the number of \textit{Outer iteration} of external method}
    \label{fig:sinkhorn_inner2}
\end{figure}

\section{On Inexact Solution of Auxiliary Subproblems}\label{InexactSolutions}





Our goal is to provide a relation between the accuracy of the solution of an optimization problem in terms of the objective residual and $\widetilde{\delta}$-`precision' in the sense of Definition \ref{def_precision}. 

\begin{theorem}
\label{lm:delta_eps} Assume that we find a point $\tilde x$ such that 
    $\phi(\tilde x) - \phi (\tilde x^*) \leq \widetilde \e$,
where $\tilde x^*$ is an exact solution of problem \eqref{min_prob}. 
Assume also that  $\phi$ has $\tilde{L}$-Lipschitz continuous gradient in $Q$. 

If $\nabla\phi(\widetilde{x}^*) = 0$, then $\tilde x = \argmin_{x \in Q}^{\widetilde{\delta}}\gav{\phi}(x)$ with 
$\widetilde{\delta} = \widetilde{R}\sqrt{2\tilde{L}\tilde{\e}}$,
where $\widetilde{R} = \max_{x,y\in Q}\|y-x\|$.

If $\phi$ is $\mu$-strongly convex on $Q$, then $\tilde x = \argmin_{x \in Q}^{\widetilde{\delta}}\gav{\phi}(x)$ with 
\begin{gather}\label{inexact}
\widetilde{\delta}= (\tilde{L}\widetilde{R}+\|\nabla\phi(\tilde{x}^*)\|_*)\sqrt{2\widetilde{\e}/\mu},
\end{gather}
\end{theorem}
\begin{proof}
1. Assume that $\nabla \phi(\tilde{x}^*)=0$. Then
\begin{align}\label{eq:Lip}
 \frac{1}{2\tilde{L}}\|\nabla \phi(\tilde{x})\|^2_* \leq   \phi(\tilde{x}) - \phi(\tilde{x}^*) \leq \widetilde{\e},
\end{align}
\begin{align}
\widetilde{\delta} = \max\limits_{x\in Q}\la \nabla \phi(\tilde{x}), \tilde{x}-x \ra \leq \|\nabla \phi(\tilde{x})\|_*\max\limits_{x\in Q}\|\tilde{x}-x\| \leq \sqrt{2\tilde{L}\widetilde{\e}}\max\limits_{x\in Q}\|\tilde x - x \|.
\end{align}
2. Let us now assume that $\nabla \phi(\tilde{x}^*)\neq 0 $.
For strongly convex function $\phi(x)$ we have
\begin{align}
    \frac{\mu}{2}\|\tilde x - \tilde x^*\|^2 \leq \phi(\tilde{x}) - \phi(\tilde{x}^*) \leq \widetilde{\e}. 
\end{align}
Hence, 
\begin{align}\label{eq_argum}
    \|\tilde x - \tilde x^*\|\leq \sqrt{\frac{2}{\mu}\widetilde{\e}}. 
\end{align}
Using this and Lipschitz gradient condition, we obtain
\begin{align}\label{eq_Lip_st}
    \|\nabla \phi(\tilde x) - \nabla \phi(\tilde x^*)\|_* \leq \tilde{L}\|\tilde x - \tilde x^*\| \leq \tilde{L}\sqrt{\frac{2}{\mu}\widetilde{\e}}.
\end{align}
Hence,
\begin{align}
    \widetilde{\delta} 
    &= \max_{x\in Q}\la \nabla\phi(\tilde{x}), \tilde x - x\ra = \max_{x\in Q}\la \nabla \phi(\tilde x) - \nabla\phi(\tilde x^*), \tilde x -x\ra +\max_{x \in Q} \la \nabla \phi(\tilde x^*), \tilde x - x\ra \notag \\
    &\overset{\eqref{eq_Lip_st}}{\leq} \tilde{L}\sqrt{\frac{2}{\mu}\widetilde{\e}}\max_{x\in Q}\|\tilde x - x\| + \max_{x \in Q}\la\nabla \phi(\tilde x^*), \tilde x^*-x\ra + \max_{x \in Q}\la \nabla \phi(\tilde x^*), \tilde x - \tilde x^*\ra \notag \\
    &\overset{\eqref{eq_argum}}{\leq} \tilde{L}\sqrt{\frac{2}{\mu}\widetilde{\e}}\max_{x\in Q}\|\tilde x - x\| + \|\nabla \phi(\tilde x^*)\|_* \sqrt{\frac{2}{\mu}\widetilde{\e}}.\notag
\end{align}
\end{proof}

\end{document}